\documentclass[11pt]{article}
\usepackage{amsfonts,amsmath,amssymb,amssymb,tikz,color, mathtools,amsthm,cite}
\usepackage{subcaption}
\usepackage{float}
\usetikzlibrary{arrows}
\usetikzlibrary{decorations.pathreplacing}

\newtheorem{theorem}{\bf Theorem}[section]
\newtheorem{corollary}[theorem]{\bf Corollary}
\newtheorem{lemma}[theorem]{\bf Lemma}
\newtheorem{proposition}[theorem]{\bf Proposition}

\newtheorem{problem}[theorem]{\bf Problem}

\theoremstyle{definition}
\newtheorem{remark}[theorem]{\bf Remark}

\newcommand{\Tr}{{\rm Tr}}

\usepackage[hyphens]{url}

\usepackage{authblk}

\title{{\bf Edge transmission irregular graphs}\footnote{Email addresses:\ \texttt{kexxu1221@126.com} (K.\ Xu), \texttt{ivan.damnjanovic@elfak.ni.ac.rs} (I.\ Damnjanovi\'{c}), \texttt{milivojevicu@pm.me} (U.\ Milivojevi\'{c}), \texttt{sandi.klavzar@fmf.uni-lj.si} (S.\ Klav\v{z}ar).}}

\author[1,2]{Kexiang Xu}
\author[3,4]{Ivan Damnjanovi\'{c}\thanks{Corresponding author.}}
\author[3]{Uro\v{s} Milivojevi\'{c}}
\author[5,6,7]{Sandi Klav\v{z}ar}

\affil[1]{School of Mathematics, Nanjing University of  Aeronautics and Astronautics,\linebreak Nanjing, Jiangsu, 210016, PR China}
\affil[2]{MIIT Key Laboratory of Mathematical Modelling and High Performance \linebreak Computing of Air Vehicles, Nanjing, Jiangsu, 210016, PR China}
\affil[3]{Faculty of Electronic Engineering, University of Ni\v{s}, Aleksandra Medvedeva 4, Ni\v{s}, 18104, Serbia}
\affil[4]{Faculty of Mathematics, Natural Sciences and Information Technologies, University of Primorska,\linebreak Glagolja\v{s}ka 8, Koper, 6000, Slovenia}
\affil[5]{Faculty of Mathematics and Physics, University of Ljubljana,\linebreak Jadranska ulica 19, Ljubljana, 1000, Slovenia}
\affil[6]{Institute of Mathematics, Physics and Mechanics, Jadranska ulica 19, Ljubljana, 1000, Slovenia}
\affil[7]{Faculty of Natural Sciences and Mathematics, University of Maribor,\linebreak Slom\v{s}kov trg 15, Maribor, 2000, Slovenia}

\date{}

\begin{document}

\maketitle

\begin{abstract}
The transmission of a vertex $v$ in a connected graph $G$ is the sum of distances from $v$ to all vertices in $G$. A transmission irregular (TI) graph is a connected graph in which any two distinct vertices have different transmissions. We extend the concept of transmission to edges by defining the transmission of an edge as the sum of the transmissions of its two endpoints. A connected graph can now be called edge transmission irregular (ETI) if any two distinct edges have different transmissions. We show that almost all graphs are not ETI and then investigate several related order realizability problems involving chemical ETI graphs. In particular, we prove that for every $n \ge 15$, there exists a subcubic tree of order $n$ that is both TI and ETI.
\end{abstract}


\medskip\noindent
\textbf{AMS Subj.\ Class.}: 05C05, 05C12, 05C92

\section{Introduction}
\label{sec:intro}

In chemical graph theory, molecules are naturally represented by chemical graphs. The graph distance function is a fundamental tool for exploring these graphs, which in turn reflect the physicochemical properties of the corresponding (organic) compounds; see~\cite{RoKi02}. As the oldest and most well-known distance-based invariant (also known as a topological index in chemical graph theory), the Wiener index~\cite{wiener-1947} of a graph $G$ is the sum of distances between all unordered pairs of vertices in $G$. Its edge version, the edge Wiener index, was later introduced in~\cite{KhYoAsWa09}. For a survey on graphs that are extremal with respect to distance-based topological indices, see~\cite{XLDGF2014}, and for a selection of recent developments with a focus on applications, see~\cite{brezo-2021, DXX26, das-2021, sorgun-2022}.

Let $G = (V(G), E(G))$ be a connected graph with vertex set $V(G)$ and edge set $E(G)$, where $n(G) \coloneqq |V(G)|$ and $m(G) \coloneqq |E(G)|$. The \textit{degree} of a vertex $v \in V(G)$, denoted by $d_G(v)$, is the number of edges incident to $v$ in $G$. For any two vertices $u, v \in V(G)$, we denote by $d_G(u, v)$ the shortest-path distance between $u$ and $v$ in $G$. A basic building block in the exploration of metric properties of a graph is the \textit{transmission} of a vertex $v \in V(G)$, which is defined as the sum of distances from $v$ to all vertices in $G$ and denoted by $\Tr_G(v)$, i.e.,
\[
    \Tr_G(v) \coloneqq \sum_{u \in V(G) } d_G(v, u) \, .
\]
The fundamental nature of this concept is highlighted by its presence in the literature under several alternative names, such as the status~\cite{abiad-2021, QZ2020} or the total distance~\cite{cava-2019, kla-2013} of a vertex. The \textit{transmission set} of $G$ is $\Tr(G) \coloneqq \{ \Tr_G(v):\ v \in V(G)\}$. If $|\Tr(G)| = n(G)$, then $G$ is called \textit{transmission irregular} (TI); see \cite{AAKS2014,xu-2021}.

In this paper, we extend this paradigm by introducing the concepts of edge transmission and edge transmission irregularity. For an edge $e = uv \in E(G)$, we define its \textit{transmission}, $\Tr_G(e)$, as the sum of the transmissions of its endpoints, i.e., $\Tr_G(e) \coloneqq \Tr_G(u) + \Tr_G(v)$. Similarly, we define the \textit{edge transmission set} of $G$ as $\mathrm{ETr}(G) \coloneqq \{\Tr_G(e):\ e \in E(G)\}$. The cardinality $|\mathrm{ETr}(G)|$ is the number of distinct edge transmissions in $G$, which can be viewed as the \textit{edge Wiener dimension} of $G$. This serves as an edge analogue of the Wiener dimension, which is defined as $|\Tr(G)|$. We define a graph $G$ to be \textit{edge transmission irregular} (ETI) if it attains the maximum possible edge Wiener dimension, i.e., if $|\mathrm{ETr}(G)| = m(G)$. Note that the sum of all edge transmissions in $G$ is equal to the well-known degree distance DD of $G$ (see~\cite{dobrynin1994, harshitha-2025}), that is:

$$\mathrm{DD}(G) = \sum_{e\in E(G)} \Tr_G(e) = \sum_{v \in V(G)} d_G(v) \, \Tr_G(v).$$

Since almost all graphs are not TI~\cite{AK2018}, the search for TI graphs has attracted the interest of several groups of researchers. A variety of structural and constructive results regarding these graphs can be found in~\cite{al-yakoob-2020, al-yakoob-2022a, al-yakoob-2022b, bezh-2021, Damnjanovic2024, DamSteAlYa2024, dobrynin-2019c, dobrynin-2019b, AXu, XK21, Xu2023}. Here, we focus on the search for ETI graphs, with an emphasis on chemical graphs.

The necessary definitions and known results are outlined 
in Section~\ref{sec:prelim} for the use in subsequent proofs. We then establish the rareness of ETI graphs in 
Section~\ref{sec:rare} and construct several infinite 
families of chemical ETI graphs in Section~\ref{sec:chemical}. Finally, Section~\ref{sec:main} completely solves the order 
realizability problem for subcubic trees that are both TI 
and ETI through the following theorem.

\begin{theorem}\label{th:main}
If $n \ge 2$, then there exists a subcubic tree of order $n$ that is both TI and ETI if and only if $n \in \{ 9, 11, 13 \}$ or $n \ge 15$.
\end{theorem}

A part of the proof of Theorem~\ref{th:main} is carried out computationally using the \texttt{SageMath}~\cite{SageMath} scripts available in \cite{GitHub}. Section~\ref{sec:conclusion} concludes the paper by providing some consequences of Theorem~\ref{th:main} regarding the TI and ETI properties of graphs and presenting a possible direction for future research.

\section{Preliminaries}
\label{sec:prelim}

All graphs considered in this paper are undirected, simple, finite and connected. For $X\subseteq V(G)$, let $G-X$ denote the subgraph of $G$ obtained by deleting the vertices of $X$ and their incident edges; we write $G - v$ instead of $G - \{v\}$, for $v \in V(G)$. Similarly, for $F \subseteq E(G)$, the graph $G-F$ is the spanning subgraph of $G$ obtained by deleting the edges of $F$, and we write $G-e$ instead of $G-\{e\}$, for $e \in E(G)$. For an edge $uv \in E(G)$, we denote by $n_G(u, v)$ the number of vertices in $G$ closer to $u$ than to $v$. The quantity $n_G(v, u)$ is defined analogously.

For any $k \in \mathbb{N}$, we denote by $[k]$ the set $\{1, \ldots, k\}$ and by $[k]_0$ the set $[k]\cup \{0\}$. We denote by $\mathcal{PS}$ the set of all perfect squares. For a set $A \subseteq \mathbb{R}$ and a number $t \in \mathbb{R}$, the coset $A + t$ is defined as $\{ a + t:\ a \in A \}$. Similarly, for sets $A, B \subseteq \mathbb{R}$, we denote by $A + B$ the set $\{ a + b:\ a \in A, b \in B \}$. A set of integers is odd (resp.\ even) if it consists of odd (resp.\ even) integers.

A vertex $v$ with $d_G(v) = 1$ is called a \textit{pendent vertex} (also a {\em leaf} if $G$ is a tree), and the edge incident with a pendent vertex is a \textit{pendent edge}. A vertex in a graph $G$ is \textit{universal} if it is adjacent to all other vertices in $G$. Now, let $P \coloneqq v_k v_{k - 1} \cdots v_1$, $k\ge 2$, with the natural adjacency relation be a path in graph $G$. If $d_G(v_k) \ge 3$, $d_G(v_1) = 1$, and $d_G(v_j) = 2$ for all $j \in \{ 2, 3, \ldots, k - 1 \}$, then $P$ is a \textit{proper pendent path} with root $v_k$. If $d_G(v_k), d_G(v_1) \ge 3$, with $\deg_G(v_j) = 2$ for all $j \in \{ 2, 3, \ldots, k - 1 \}$, then $P$ is an \textit{internal path} with terminals $v_k$ and $v_1$.

The exact definition of a chemical graph varies in the literature. Depending on the use case, some authors define a \emph{chemical graph} as a connected graph with a maximum degree of at most four, while others restrict the maximum degree to at most three. In this paper, we adopt the more general convention and allow the maximum degree to be at most four. Nonetheless, to conform to both standards, we investigate both subcubic and subquartic connected graphs, placing a particular emphasis on trees.

A vertex of degree at least three is a called a \textit{branching vertex}. A tree with a unique branching vertex is \textit{starlike}. We denote a starlike tree $T$ with branching vertex $v$ by $S(n_1, n_2, \ldots, n_k)$ if $T - v$ consists of $k$ disjoint paths of orders $n_1, n_2, \ldots, n_k$, respectively. The pendent path of length $n_j$ from $v$ is called an {\em $n_j$-arm} of $T$.

For any $n \ge 3$ and $k \ge 2$, let $C_n(a_1, a_2, \ldots, a_k)$ be the \textit{ordinary caterpillar tree} of order $n + k$ that arises from the path $P_n \coloneqq v_1 v_2 \cdots v_n$ by attaching a leaf to each of the vertices $v_{a_1}, v_{a_2}, \ldots, v_{a_k}$, with $2 \le a_j \le n - 1$ for each $j \in [k]$. Also, for any $n \ge 3$ and $k \ge 2$, let $C_n(a_1, b_1; a_2, b_2; \ldots; a_k, b_k)$ be the \textit{variant caterpillar tree} of order $n + \sum_{j = 1}^k b_j$ that arises from the path $P_n \coloneqq v_1 v_2 \cdots v_n$ by attaching a pendent path of length $b_j$ to the vertex $v_{a_j}$, with $2 \le a_j \le n - 1$ and $b_j \ge 1$ for each $j \in [k]$.

\begin{lemma}[\hspace{1sp}{\cite{Bala}}]\label{equal}
For any graph $G$ and edge $uv \in E(G)$, we have $\Tr_G(u) -\Tr_G(v) = n_G(v, u) - n_G(u, v)$.
\end{lemma}

\begin{lemma}[\hspace{1sp}{\cite{XK21}}]\label{pend}
If $G$ is a graph of order $n$ and $v_0 v_1 \cdots v_k$ a pendent path in $G$ with root $v_0$ and leaf $v_k$, then $\Tr_G(v_j) - \Tr_G(v_0) = j(j + n - 2k - 1)$ for all $j \in [k]_0$.
\end{lemma}

\section{Rareness of ETI graphs}
\label{sec:rare}

In this section we show the rareness of ETI graphs. Before doing so, we require the following concept. For an edge $uv \in E(G)$, the \textit{degree weight} $\mathrm{dw}_G(uv)$ is the sum of the degrees of its endpoints, i.e., $\mathrm{dw}_G(uv) \coloneqq d_G(u) + d_G(v)$. Next, we prove a property of the degree weights of edges that is analogous to the well-known property of vertex degrees.

\begin{lemma}\label{dw} 
Let $G$ be a graph with at least two edges. Then $G$ has two distinct edges of the same degree weight.
\end{lemma}
\begin{proof}
Let $L(G)$ be the line graph of $G$. For any edge $uv \in E(G)$, the corresponding vertex in $L(G)$ has 
degree $d_G(u) + d_G(v) - 2$. Since $L(G)$ has at least two vertices, it has two distinct vertices of the same degree.
Therefore, $G$ has two distinct edges of the same degree weight.
\end{proof}

Below, we establish the rareness of ETI graphs using a similar idea as in \cite{AK2018}, mirroring the known behavior of TI graphs.

\begin{theorem}\label{rare-eti}
Almost all graphs are not ETI.
\end{theorem}
\begin{proof}
By a standard probabilistic argument, almost all graphs have diameter two. Thus, it suffices to prove that no graph $G$ with diameter two is ETI. For any vertex $v \in V(G)$, its transmission can be expressed as $\Tr_G(v) = d_G(v) + 2(n(G) - 1 - d_G(v)) = 2n(G) - 2 - d_G(v)$. It follows that the transmission of any edge $uv \in E(G)$ is given by $\Tr_G(uv) = 4n(G) - 4 - (d_G(u) + d_G(v))$. By Lemma~\ref{dw}, the graph $G$ has two distinct edges with the same transmission.
\end{proof}

In light of the rareness of ETI graphs demonstrated in Theorem~\ref{rare-eti}, it is of interest to characterize ETI graphs within specific classes of graphs. Since both TI and ETI graphs are rare, the class of graphs that are simultaneously TI and ETI is sparse. Furthermore, these two structural properties do not imply one another.

Excluding the trivial case $K_1$, an exhaustive search using the \texttt{geng} utility from the \texttt{nauty} package \cite{Nauty} confirms that the minimum order for a TI graph is $7$. There are exactly three such graphs of order $7$, as shown in Fig.~\ref{fig:smallest_ti}. The graph from Fig.~\ref{fig:smallest_ti_1} is $S(1, 2, 3)$ as the only tree among them and hence the unique smallest TI tree. The graph from Fig.~\ref{fig:smallest_ti_2} is the only one among them whose vertex transmissions form a sequence of consecutive integers as the unique smallest interval transmission irregular (ITI) graph \cite{al-yakoob-2022a}. The graph from Fig.~\ref{fig:smallest_ti_3} is the unique smallest unicyclic TI graph. However, none of these three graphs is ETI.

If we exclude $K_1$ and $K_2$ as trivial ETI graphs, the smallest ETI graphs appear at order~$8$. Among the five such graphs of order $8$, exactly one is simultaneously TI and ETI; see Fig.~\ref{fig:smallest_ti_eti}. The remaining four are ETI, but not TI, as shown in Fig.~\ref{fig:smallest_eti_not_ti}. The graph from Fig.~\ref{fig:smallest_eti_not_ti_1} is $S(1, 2, 4)$ as the only tree among all five graphs, and hence the unique smallest ETI tree. Out of the last three graphs, two are unicyclic, whereas the third is neither a tree nor unicyclic.

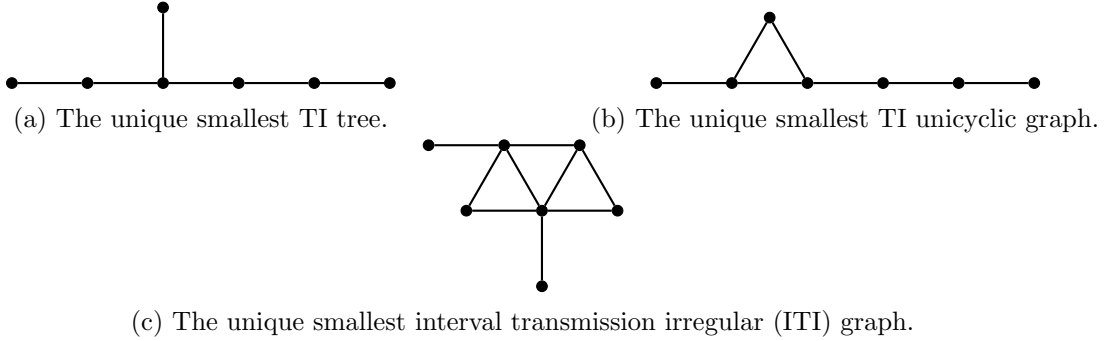
\begin{figure}[H]
\centering

\begin{subfigure}[b]{0.45\textwidth}
\centering
\begin{tikzpicture}[scale=1.0]
\tikzstyle{vertex}=[draw,circle,minimum size=4pt,inner sep=1pt,fill=black]
\tikzstyle{edge}=[draw,thick]
\tikzstyle{dedge}=[draw,thick,dashed]

\node[vertex] (0) at (0, 0) {};
\node[vertex] (1) at (1, 0) {};
\node[vertex] (2) at (2, 0) {};
\node[vertex] (3) at (3, 0) {};
\node[vertex] (4) at (4, 0) {};
\node[vertex] (5) at (5, 0) {};
\node[vertex] (6) at (2, 1) {};

\path[edge] (0) -- (1);
\path[edge] (1) -- (2);
\path[edge] (2) -- (3);
\path[edge] (3) -- (4);
\path[edge] (4) -- (5);
\path[edge] (6) -- (2);
\end{tikzpicture}
\caption{The unique smallest TI tree.}
\label{fig:smallest_ti_1}
\end{subfigure}
\hfill
\begin{subfigure}[b]{0.45\textwidth}
\centering
\begin{tikzpicture}[scale=1.0]
\tikzstyle{vertex}=[draw,circle,minimum size=4pt,inner sep=1pt,fill=black]
\tikzstyle{edge}=[draw,thick]
\tikzstyle{dedge}=[draw,thick,dashed]

\node[vertex] (0) at (0, 0) {};
\node[vertex] (1) at (1, 0) {};
\node[vertex] (2) at (2, 0) {};
\node[vertex] (3) at (3, 0) {};
\node[vertex] (4) at (4, 0) {};
\node[vertex] (5) at (5, 0) {};
\node[vertex] (6) at (1.5, 0.866) {};

\path[edge] (0) -- (1);
\path[edge] (1) -- (2);
\path[edge] (2) -- (3);
\path[edge] (3) -- (4);
\path[edge] (4) -- (5);
\path[edge] (6) -- (1);
\path[edge] (6) -- (2);
\end{tikzpicture}
\caption{The unique smallest TI unicyclic graph.}
\label{fig:smallest_ti_3}
\end{subfigure}
\\
\begin{subfigure}[b]{0.90\textwidth}
\centering
\begin{tikzpicture}[scale=1.0]
\tikzstyle{vertex}=[draw,circle,minimum size=4pt,inner sep=1pt,fill=black]
\tikzstyle{edge}=[draw,thick]
\tikzstyle{dedge}=[draw,thick,dashed]

\node[vertex] (0) at (0, 0) {};
\node[vertex] (1) at (1, 0) {};
\node[vertex] (2) at (2, 0) {};
\node[vertex] (3) at (0.5, 0.866) {};
\node[vertex] (4) at (1.5, 0.866) {};
\node[vertex] (5) at (1, -1) {};
\node[vertex] (6) at (-0.5, 0.866) {};

\path[edge] (0) -- (1);
\path[edge] (1) -- (2);
\path[edge] (3) -- (4);
\path[edge] (0) -- (3);
\path[edge] (3) -- (1);
\path[edge] (1) -- (4);
\path[edge] (4) -- (2);
\path[edge] (5) -- (1);
\path[edge] (6) -- (3);
\end{tikzpicture}
\caption{The unique smallest interval transmission irregular (ITI) graph.}
\label{fig:smallest_ti_2}
\end{subfigure}

\caption{The three smallest TI graphs, with $K_1$ excluded as trivial.}
\label{fig:smallest_ti}
\end{figure}

\begin{figure}[H]
\centering
\begin{tikzpicture}[scale=1.0]
\tikzstyle{vertex}=[draw,circle,minimum size=4pt,inner sep=1pt,fill=black]
\tikzstyle{edge}=[draw,thick]
\tikzstyle{dedge}=[draw,thick,dashed]

\node[vertex] (0) at (0, 0) {};
\node[vertex] (1) at (1, 0) {};
\node[vertex] (2) at (2, 0) {};
\node[vertex] (3) at (3, 0) {};
\node[vertex] (4) at (4, 0) {};
\node[vertex] (5) at (5, 0) {};
\node[vertex] (6) at (6, 0) {};
\node[vertex] (7) at (2.5, 0.866) {};

\path[edge] (0) -- (1);
\path[edge] (1) -- (2);
\path[edge] (2) -- (3);
\path[edge] (3) -- (4);
\path[edge] (4) -- (5);
\path[edge] (5) -- (6);
\path[edge] (7) -- (2);
\path[edge] (7) -- (3);
\end{tikzpicture}
\caption{The unique smallest graph that is both TI and ETI, with $K_1$ excluded as trivial.}
\label{fig:smallest_ti_eti}
\end{figure}
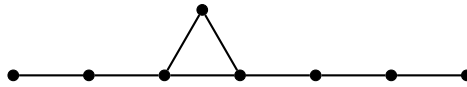

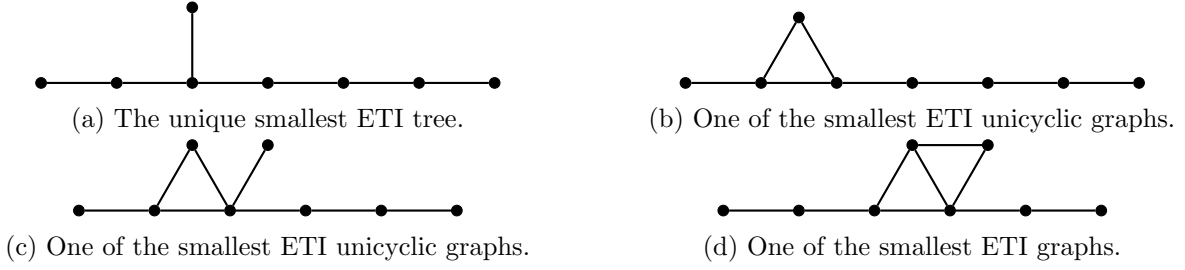
\begin{figure}[H]
\centering

\begin{subfigure}[b]{0.45\textwidth}
\centering
\begin{tikzpicture}[scale=1.0]
\tikzstyle{vertex}=[draw,circle,minimum size=4pt,inner sep=1pt,fill=black]
\tikzstyle{edge}=[draw,thick]
\tikzstyle{dedge}=[draw,thick,dashed]

\node[vertex] (0) at (0, 0) {};
\node[vertex] (1) at (1, 0) {};
\node[vertex] (2) at (2, 0) {};
\node[vertex] (3) at (3, 0) {};
\node[vertex] (4) at (4, 0) {};
\node[vertex] (5) at (5, 0) {};
\node[vertex] (6) at (6, 0) {};
\node[vertex] (7) at (2, 1) {};

\path[edge] (0) -- (1);
\path[edge] (1) -- (2);
\path[edge] (2) -- (3);
\path[edge] (3) -- (4);
\path[edge] (4) -- (5);
\path[edge] (5) -- (6);
\path[edge] (7) -- (2);
\end{tikzpicture}
\caption{The unique smallest ETI tree.}
\label{fig:smallest_eti_not_ti_1}
\end{subfigure}
\hfill
\begin{subfigure}[b]{0.45\textwidth}
\centering
\begin{tikzpicture}[scale=1.0]
\tikzstyle{vertex}=[draw,circle,minimum size=4pt,inner sep=1pt,fill=black]
\tikzstyle{edge}=[draw,thick]
\tikzstyle{dedge}=[draw,thick,dashed]

\node[vertex] (0) at (0, 0) {};
\node[vertex] (1) at (1, 0) {};
\node[vertex] (2) at (2, 0) {};
\node[vertex] (3) at (3, 0) {};
\node[vertex] (4) at (4, 0) {};
\node[vertex] (5) at (5, 0) {};
\node[vertex] (6) at (6, 0) {};
\node[vertex] (7) at (1.5, 0.866) {};

\path[edge] (0) -- (1);
\path[edge] (1) -- (2);
\path[edge] (2) -- (3);
\path[edge] (3) -- (4);
\path[edge] (4) -- (5);
\path[edge] (5) -- (6);
\path[edge] (7) -- (1);
\path[edge] (7) -- (2);
\end{tikzpicture}
\caption{One of the smallest ETI unicyclic graphs.}
\end{subfigure}
\\
\begin{subfigure}[b]{0.45\textwidth}
\centering
\begin{tikzpicture}[scale=1.0]
\tikzstyle{vertex}=[draw,circle,minimum size=4pt,inner sep=1pt,fill=black]
\tikzstyle{edge}=[draw,thick]
\tikzstyle{dedge}=[draw,thick,dashed]

\node[vertex] (0) at (0, 0) {};
\node[vertex] (1) at (1, 0) {};
\node[vertex] (2) at (2, 0) {};
\node[vertex] (3) at (2.5, 0.866) {};
\node[vertex] (4) at (3, 0) {};
\node[vertex] (5) at (4, 0) {};
\node[vertex] (6) at (5, 0) {};
\node[vertex] (7) at (1.5, 0.866) {};

\path[edge] (0) -- (1);
\path[edge] (1) -- (2);
\path[edge] (2) -- (3);
\path[edge] (2) -- (4);
\path[edge] (4) -- (5);
\path[edge] (5) -- (6);
\path[edge] (7) -- (1);
\path[edge] (7) -- (2);

\end{tikzpicture}
\caption{One of the smallest ETI unicyclic graphs.}
\end{subfigure}
\hfill
\begin{subfigure}[b]{0.45\textwidth}
\centering
\begin{tikzpicture}[scale=1.0]
\tikzstyle{vertex}=[draw,circle,minimum size=4pt,inner sep=1pt,fill=black]
\tikzstyle{edge}=[draw,thick]
\tikzstyle{dedge}=[draw,thick,dashed]

\node[vertex] (0) at (0, 0) {};
\node[vertex] (1) at (1, 0) {};
\node[vertex] (2) at (2, 0) {};
\node[vertex] (3) at (3, 0) {};
\node[vertex] (4) at (4, 0) {};
\node[vertex] (5) at (5, 0) {};
\node[vertex] (6) at (2.5, 0.866) {};
\node[vertex] (7) at (3.5, 0.866) {};

\path[edge] (0) -- (1);
\path[edge] (1) -- (2);
\path[edge] (2) -- (3);
\path[edge] (3) -- (4);
\path[edge] (4) -- (5);
\path[edge] (6) -- (2);
\path[edge] (6) -- (3);
\path[edge] (7) -- (3);
\path[edge] (7) -- (6);
\end{tikzpicture}
\caption{One of the smallest ETI graphs.}
\end{subfigure}

\caption{Four graphs non-TI among the five smallest ETI graphs, with $K_1$ and $K_2$ excluded as trivial.}
\label{fig:smallest_eti_not_ti}
\end{figure}

\section{Chemical ETI graphs}
\label{sec:chemical}

In this section we consider the chemical ETI graphs of even orders. 
For an even integer $n\geq 10$, we denote by $H_n$ (see Fig.~\ref{fig:H-n}) a tree obtained by attaching a pendent path of length $\frac{n}{2}-2$ to the second vertex of $P_5$ and another pendent path of length $\frac{n}{2}-3$ to the fourth vertex of $P_5$.

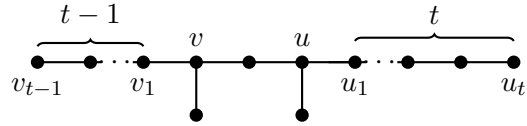
\begin{figure}[ht!]
\begin{center}
\begin{tikzpicture}[scale=0.7,style=thick]
\tikzstyle{every node}=[draw=none,fill=none]
\def\vr{3pt} 

\begin{scope}[yshift = 0cm, xshift = 0cm]
\path (0,0) coordinate (x1);
\path (1,0) coordinate (x2);
\path (2,0) coordinate (x21);
\path (2,-1) coordinate (y21);
\path (3,0) coordinate (x3);
\path (6,0) coordinate (x7);
\path (-1,0) coordinate (x4);
\path (0,-1) coordinate (x6);
\path (4,0) coordinate (z1);
\path (5,0) coordinate (z2);
\path (-2,0) coordinate (z3);
\path (-3,0) coordinate (z4);
\draw (x1) -- (x2) -- (x21) -- (3.2,0);
\draw (x1) -- (x6);
\draw (x21) -- (y21);
\draw (x1) -- (-1.2,0);
\draw (x7) -- (z2) -- (3.8,0);
\draw (z4) -- (-1.8,0);
\draw (2.5,0) -- (3.0,0);

\draw (x1)  [fill=black] circle (\vr);
\draw (x2)  [fill=black] circle (\vr);
\draw (x21)  [fill=black] circle (\vr);
\draw (y21)  [fill=black] circle (\vr);
\draw (x3)  [fill=black] circle (\vr);
\draw (x7)  [fill=black] circle (\vr);
\draw (x4)  [fill=black] circle (\vr);

\draw (x6)  [fill=black] circle (\vr);
\draw (z1)  [fill=black] circle (\vr);
\draw (z2)  [fill=black] circle (\vr);
\draw (z3)  [fill=black] circle (\vr);
\draw (z4)  [fill=black] circle (\vr);
\draw (3.5,0) node {$\cdots$};
\draw (-1.5,0) node {$\cdots$};
\draw (4.5,0.9) node {$t$};
\draw (-2,0.9) node {$t-1$};

\draw [above] (x1)++(0, 0.1) node {$v$};
\draw [above] (x21)++(0, 0.1) node {$u$};
\draw [below] (x3)++(0, -0.1) node {$u_1$};
\draw [below] (x7)++(0, -0.1) node {$u_t$};
\draw [below] (x4)++(0, -0.1) node {$v_1$};
\draw [below] (z4)++(0, -0.1) node {$v_{t-1}$};

\draw [decorate, decoration = {brace}] (3,0.4) --  (6,0.4);
\draw [decorate, decoration = {brace}] (-3,0.3) --  (-1,0.3);
\end{scope}
\end{tikzpicture}
\end{center}
\caption{The tree $H_n$ of order $n=2t+4$ with $t\geq 3$.}
\label{fig:H-n}
\end{figure}

\begin{proposition}\label{H-n}
If $n\geq 10$ is an even integer and $\mathcal{PS} \cap \big(\{4,12,20\}+2n\big) = \emptyset$, then $H_n$ is ETI.
\end{proposition}
\proof Assume that $n=2t+4$ with $t\geq 3$ and $u$, $v$ are two vertices of degree $3$ in $H:=H_n$. From the definition of $H_n$, we further assume that $uu_1u_2\cdots u_t$ and $vv_1v_2\cdots v_{t-1}$ are two proper pendent paths beginning at $u$ and $v$, respectively, in $T$, and $u_0$, $v_0$ are two leaves adjacent to $u$ and $v$, respectively, such that $N_H(u)\cap N_H(v)=\{w\}$.

From the structure of $H$ and Lemma \ref{equal}, we have $\Tr(u)=Tr(w):=x$. In view of Lemmas \ref{pend} and \ref{equal}, we have $\Tr(v)=x+2$ with $\Tr(v_0)=x+2t+4$, $\Tr(v_j)=x+j^2+5j+2$ for $j\in [t-1]$ and $\Tr(u_0)=x+2t+2$ with $\Tr(u_i)=x+i^2+3i$ for $i\in [t]$. By the definition of edge transmission, we have $\mathrm{ETr}(H)=(A_0\cup A_1\cup A_2)+2x$ where \begin{align*}
A_0 & = \{0,2,2t+2,2t+6\}, \\
A_1 & = \{2(m+2)^2-4:\ m\in[t-1]_0\}, \\
A_2 & = \{2(m+1)(m+5):\ m\in [t-2]_0\}.
\end{align*}
We first prove $A_1\cap A_2=\emptyset$. To the contrary, suppose that $2(i+2)^2-4=2(j+1)(j+5)$ with $i\in [t-1]_0$ and $j \in [t-2]_0$, that is, $i^2+4i=j^2+6j+3$.
It follows that $i^2-j^2=6j-4i+3$ with $i>j$ of different parities. If $i-j=1$, we have $2j+1=2j-1$ as a clear contradiction.
If $i-j\geq 3$, we have $6j+9\leq 6j-4(j+3)+3$, that is, $6j+9\leq 2j-9$, yielding that $j\leq -\frac{9}{2}$, which contradicts the range of $j$. Thus $A_1\cap A_2=\emptyset$ holds as desired.

Note that $\{0,2\}\cap (A_1\cup A_2)=\emptyset$. If $2t+2\in A_1$, then there exists an integer $m\in [t-1]_0$ with $2t+2=2(m^2+4m+2)$, that is, $(m+2)^2=t+3$. Thus $4(m+2)^2=4t+12=2n+4$, contradicting the assumption that $2n+4\notin \mathcal{PS}$. Since $2n+12\notin \mathcal{PS}$, that is, $4t+20$ is not a perfect square, we have $2t+2\notin A_2$ (If not, $2t+2=2(m+1)(m+5)$ with $m\in [t-2]_0$, which implies that $(m+3)^2=t+5$, that is, $4t+20=2n+12=4(m+3)^2$ as a contradiction). Therefore $2t+2\notin A_1\cup A_2$. Similarly as above, $2t+6\notin A_1\cup A_2$ follows from the assumption that $\mathcal{PS}\cap \left( \{12,20\}+2n \right) = \emptyset$. Thus $A_0\cup A_1\cup A_2$ forms a partition set, that is, $H_n$ is ETI.
\qed
\medskip

Denote by $C_3^*(a,b,c)$ a graph obtained by attaching three pendent paths of lengths $a$, $b$ and $c$ to the three vertices of the triangle $C_3$, respectively. Next we show the existence of ETI graphs of order $n$ with $\left( \{4,12,20\}+2n \right) \cap \mathcal{PS} \neq \emptyset$. Before doing it, we give an equivalent statement to the fact $\left( \{4,12,20\}+2n \right) \cap \mathcal{PS} \neq \emptyset$.

\begin{lemma}
\label{equi} 
If $n\geq 10$ is an integer, then the following properties hold. 
\begin{enumerate}
\item[(i)] $2n+4 \in \mathcal{PS}$ if and only if $n\in \{2m(m+4)+6:\ m\in \mathbb{N}\}$.
\item[(ii)] $2n+12 \in \mathcal{PS}$ if and only if $n\in \{2m(m+4)+2:\ m\in \mathbb{N}\}$.
\item[(iii)] $2n+20 \in \mathcal{PS}$ if and only if $n\in \{2m(m+6)+8:\ m\in \mathbb{N}\}$. 
\end{enumerate}
\end{lemma}


\proof 
We prove statement (i); the other two statements are proved analogously and the proofs are hence omitted. So let $n\ge 10$ be an integer. It can be routinely checked that if $n = 2m(m+4)+6$ for some $m\in \mathbb{N}$, then $2n+4\in \mathcal{PS}$. Conversely, assume that $2n+4\in \mathcal{PS}$, where $n\geq 10$. Then $2n+4=(2m+4)^2$ and hence $n=2m(m+4)+6$ with $m\in \mathbb{N}$. 
\qed
\medskip

\begin{proposition}
\label{c1} 
If $n\geq 10$ is even and $2n+12\in\mathcal{PS}$, then there is an ETI graph of order~$n$.
\end{proposition}

\proof 
By Lemma~\ref{equi}(ii) we have $n=2m(m+4)+2$ for some $m \in \mathbb{N}$. We will now distinguish between the following two cases. 

\medskip\noindent
{\bf Case 1}: $m \equiv 1 \pmod 3$.

In this case, $m=3a+1$ with $a\geq 0$. Then $n=3x+3$ with $x=3(2a^2+4a+1)$. Let $G=C_3^*(x+1,x,x-1)$ with $v_i\in V(G)$ for $i\in [3]$ at which a pendent path of length $x+2-i$ is attached. Assume that $\Tr(v_1)=g$. From the structure of $G$ and Lemma \ref{equal}, we have $\Tr(v_2)=g+1$, $\Tr(v_3)=g+2$ and $\Tr(G)-g=A_1\cup A_2\cup A_3\cup\{0,1,2\}$ where \begin{align*}
A_1 & = \{tx+t^2:\ t\in[x+1]\}, \\
A_2 & = \{tx+(t+1)^2:\ t\in[x]\}, \\
A_3 & = \{tx+t(t+4)+2:\ t\in[x-1]\}.
\end{align*}
By the definition of edge transmission, we have $\mathrm{ETr}(G)-2g=E_1\cup E_2\cup E_3\cup \{1,2,3\}$ with \begin{align*}
E_1 & = \{(2s-1)x+(s-1)^2+s^2:\ s\in[x+1]\}, \\
E_2 & = \{(2s-1)x+(s+1)^2+s^2:\ s\in[x]\}, \\
E_3 & = \{(2s-1)x+2s^2+6s+1:\ s\in[x-1]\}.
\end{align*}
Note that $\min E_i> 3$ for $i\in [3]$. Next we show that $E_p\cap E_q=\emptyset$ for any distinct $p,q\in [3]$. If $(2i-1)x+(i-1)^2+i^2=(2j-1)x+j^2+(j+1)^2$ with $i\in [x+1]$ and $j\in [x]$, then $i>j$. It follows that $(i-j)x=(i+j)(j-i+1)>0$, yielding that $i-1<j<i$ as a contradiction. Thus $E_1\cap E_2=\emptyset$. If there are two integers $i\in [x+1]$ and $j\in [x-1]$ such that $(2i-1)x+(i-1)^2+i^2=(2j-1)x+2j^2+6j+1$, then $(i-j)(x+i+j)=3j+i$ with $i>j$. If $i-j=1$, we have $x=2j$ as a contradiction to the parity of $x$. If $i-j\geq 2$, we have $3j+i\geq 2(i+j+3)$, that is, $j\geq i+6$, contradicting the fact that $i>j$. Therefore $E_1\cap E_3=\emptyset$. Similarly as above, we have $E_2\cap E_3=\emptyset$. Thus $G=C_3^{*}(x+1,x,x-1)$ is ETI.

\medskip\noindent
{\bf Case 2}: $m \bmod 3 \in \{0,2\}$.

Assume first that $m=3b$. Then $n=3y-1$ with $y=6b^2+8b+1$ where $b\geq 1$. Note that $y\geq 15$. Let $H=C_3^*(y,y-1,y-3)$ with $u_i\in V(H)$ at which a pendent path of length $y+1-i$ with $i\in [2]$ and $u_3\in V(H)$ at which a pendent path of length $y-3$ is attached. Assume that $\Tr_H(u_1)=h$. By Lemma \ref{equal} and the structure of $H$, we have $\Tr(u_2)=h+1$ and $\Tr(u_3)=h+3$ with $\Tr(H)-h=\{0,1,3\}\cup B_1\cup B_2\cup B_3$ where \begin{align*}
B_1 & = \{ty+t^2-2t:\ t\in[y]\}, \\
B_2 & = \{ty+t^2+1:\ t\in[y-1]\}, \\
B_3 & = \{ty+t(t+4)+3:\ t\in[y-3]\}.
\end{align*}  By the definition of edge transmission, we have $\mathrm{ETr}(H)-2h=F_1\cup F_2\cup F_3\cup \{1,3,4\}$ with \begin{align*}
F_1 & = \{(2s-1)y+2s^2-6s+3:\ s\in[y]\}, \\
F_2 & = \{(2s-1)y+2s^2-2s+3:\ s\in[y-1]\}, \\
F_3 & = \{(2s-1)y+2s^2+6s+3:\ s\in[y-3]\}.
\end{align*}
Note that $\min \bigcup\limits_{i=1}^{3}F_i>4$. If there exist two equal numbers $(2p-1)y+2p^2-6p+3\in F_1$ and $(2q-1)y+2q^2-2q+3\in F_2$, then $(p-q)(p+q+y)=3p-q$ with $p>q$. If $p-q=1$, we get $y=2$, contradicting the range of $y$. If $p-q=2$, we have $y+q=1$ as a clear contradiction. If $p-q\geq 3$, we have $3p-q\geq 3p+3q+3y$, that is, $4q+3y\leq 0$, a contradiction again. Thus $F_1\cap F_2=\emptyset$. For two equal numbers $(2p-1)y+2p^2-6p+3\in F_1$ and $(2q-1)y+2q^2+6q+3\in F_3$, we have $(p-q)(p+q+y)=3(p+q)$ with $p>q$. If $p-q=1$, we have $y=4q+2$, contradicting the parity of $y$ with $y=6b^2+8b+1$. If $p-q\geq 2$, we have $p+q\geq 2y$, contradicting the respective ranges of $p$ and $q$ with $p\leq y$ and $q\leq y-3$. Therefore $F_1\cap F_3=\emptyset$. Similarly as above, we have $F_2\cap F_3=\emptyset$. Thus $H$ is ETI.

 Assume second that $m=3c+2$. Then $n=3z-1$ with $z=6c^2+16c+9$. We consider the graph $J=C_3^*(z,z-1,z-3)$. By a very analogous reasoning to that in the proof  of the above case $m=3b$, we can prove that $J$ is ETI and we are done. 
\qed
\medskip

Denote by $W$ a graph obtained by joining with an edge one vertex of a triangle and one vertex of another triangle. For a positive integer $m$ with $n=2m(m+4)+6$, we denote by $W(a,b;c)$ a graph obtained from $W$ by attaching a pendent path of length $a$ at a $2$-degree vertex, say $u$, and a pendent path of length $b$ at the $2$-degree neighbor of $u$, and a pendent path of length $c$ at one of the remaining $2$-degree vertices. See in Fig.~\ref{fig:W} the graph $W(3,4;3)$. Next we show the attainable property of $n$ on ETI graphs with even  $n\geq 10$ with $2n+4\in\mathcal{PS}$.

\begin{figure}[ht!]
\begin{center}
\begin{tikzpicture}[scale=0.7,style=thick]
\tikzstyle{every node}=[draw=none,fill=none]
\def\vr{3pt} 

\begin{scope}[yshift = 0cm, xshift = 0cm]
\path (1,0) coordinate (x1);
\path (2,0) coordinate (x2);
\path (1.5,-1) coordinate (y1);
\path (0,0) coordinate (z1);
\path (3,0) coordinate (x3);
\path (-1,0) coordinate (x4);
\path (-2,0) coordinate (x5);
\path (-0.5,-1) coordinate (z2);
\path (-1.5,-1) coordinate (z3);
\path (-2.5,-1) coordinate (z4);
\path (-3.5,-1) coordinate (z5);
\path (-4.5,-1) coordinate (z6);
\path (4,0) coordinate (x6);
\path (5,0) coordinate (x7);
\path (-4,0) coordinate (x8);
\path (-3,0) coordinate (x10);
\draw (x4) -- (x5) -- (x10) -- (x8);
\draw (x1) -- (x3) -- (x6) -- (x7);
\draw (x1) -- (x2) -- (y1) -- (x1);
\draw (x1) -- (z1);
\draw (z1) -- (x4) -- (z2) -- (z1);
\draw (z2) -- (z3) -- (z4) -- (z5) -- (z6);
\draw (x1)  [fill=black] circle (\vr);
\draw (x2)  [fill=black] circle (\vr);
\draw (y1)  [fill=black] circle (\vr);
\draw (z1)  [fill=black] circle (\vr);
\draw (z2)  [fill=black] circle (\vr);
\draw (z3)  [fill=black] circle (\vr);
\draw (z4)  [fill=black] circle (\vr);
\draw (z5)  [fill=black] circle (\vr);
\draw (z6)  [fill=black] circle (\vr);
\draw (x3)  [fill=black] circle (\vr);
\draw (x4)  [fill=black] circle (\vr);
\draw (x5)  [fill=black] circle (\vr);
\draw (x6)  [fill=black] circle (\vr);
\draw (x7)  [fill=black] circle (\vr);
\draw (x8)  [fill=black] circle (\vr);
\draw (x10)  [fill=black] circle (\vr);

\draw (-1, 0.4) node {$u_3$};
\draw  (0,0.4) node {$u$}; 
\draw (1,0.4) node {$u_1$};
\draw [below]  (-0.5,-1.) node {$u_2$};

\end{scope}
\end{tikzpicture}
\end{center}
\caption{The graph  $W(3,4;3)$.}
\label{fig:W}
\end{figure}
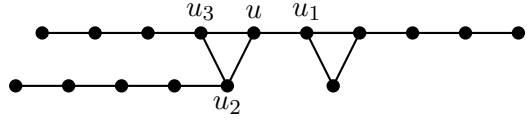

\begin{proposition}
\label{c2} 
If $n\geq 10$ is even and $2n+4\in \mathcal{PS}$, then there is an ETI graph of order~$n$.
\end{proposition}

\proof 
By Lemma~\ref{equi} we know that $n=2m(m+4)+6$ for some $m \in \mathbb{N}$. We again distinguish the following two cases. 

\medskip\noindent
{\bf Case 1}: $m \equiv 1 \pmod 3$.

In this case, $m=3a+1$. Then $n=3x+1$ with $x=6a^2+12a+5$. Let $G=W(x-2,x-1; \linebreak x-2)$ with $u\in V(G)$  of degree $3$ with $N_G(u)=\{u_1,u_2,u_3\}$ where $u_1$ lies in a triangle and at $u_i$ a pendent path of length $x-i+1$ is attached for $i\in\{2,3\}$. Now we prove that $G$ is ETI. Assume that $\Tr(u)=g$  and $N_G(u_1)\setminus\{u\}=\{v_1,v_2\}$ with $d_G(v_1)=2$. Then, in $G$, at $v_2$ a pendent path of length $x-2$ is attached.  From the structure of $G$ and Lemma \ref{equal}, we have $\Tr(u_1)-g=x-1$, $\Tr(u_i)-g=i$ for $i\in\{2,3\}$, $\Tr(v_1) - \Tr(u_1) = 2x$ and $\Tr(v_2) - \Tr(u_1)=x+2$. Then $\Tr(G)-g=A_1\cup A_2\cup A_3\cup\{0,2,3,x-1,3x-1,2x+1\}$ where \begin{align*}
A_1 & = \{(k+2)x+k^2+4k+1:\ k\in[x-2]\}, \\
A_2 & = \{kx+(k+1)^2+1:\ k\in[x-1]\}, \\
A_3 & = \{kx+(k+1)(k+3):\ k\in[x-2]\}.
\end{align*}
By the definition of edge transmission, we have $\mathrm{ETr}(G)-2g=\bigcup\limits_{i=0}^{3}B_i$ with \begin{align*}
B_0 & =\{2,3,5,x-1,3x,4x-2,5x\}, \\
B_1 & = \{(2k+3)x+2k^2+6k-1:\ k\in[x-2]\}, \\
B_2 & = \{(2k-1)x+2k^2+2k+3:\ k\in[x-1]\}, \\
B_3 & = \{(2k-1)x+2k^2+6k+3:\ k\in[x-2]\}.
\end{align*}
Note that $\min B_1=5x+7>5x=\max B_0$. So $B_0\cap B_1=\emptyset$. Since $x\geq 5$ is odd, we have $B_0\cap (B_2\cup B_3)=\emptyset$. To prove the ETI property of $G$, we only need to prove $B_i\cap B_j=\emptyset$ for any distinct $i,j\in [3]$. If there are two integers $p\in [x-2]$ and $q\in [x-1]$ such that $(2p+3)x+2p^2+6p-1=(2q-1)x+2q^2+2q+3$, then $(q-p)(x+q+p+1)=2(x-1+p)$  with $q>p$. If $q-p=1$, we have $x=4$, contradicting the range of $x$. If $q-p\geq 2$, then $2x-2+2p\geq 2(x+2p+3)$, yielding $p\leq -4$ as a clear contradiction. Therefore $B_1\cap B_2=\emptyset$. If there exist two integers $p,q\in [x-2]$ with $(2p+3)x+2p^2+6p-1=(2q-1)x+2q^2+6q+3$, then $(q-p)(x+q+p+3)=2x-2$ with $q>p$. If $q-p=1$, we have $x=2p+6$ as a clear contradiction to the parity of $x$. If $q-p\geq 2$, then $2x-2\geq 2(x+2p+5)$, that is, $p\leq -3$ as a contradiction. Thus $B_1\cap B_3=\emptyset$. If there are two integers $p\in [x-1]$ and $q\in [x-2]$ with $(2p-1)x+2p^2+2p+3=(2q-1)x+2q^2+6q+3$, then $(p-q)(x+p+q)=3q-p$ with $p>q$. It follows that $3q-p\geq 5+p+q$. Then $q-p\geq \frac{5}{2}$, that is, $p-q\leq -\frac{5}{2}$, as a contradiction again, implying that $B_2\cap B_3=\emptyset$. Therefore  $G=W(x-2,x-1;x-2)$ is ETI as desired.

\medskip\noindent
{\bf Case 2}: $m \bmod 3 \in \{0,2\}$.
 
 If $m=3b+2$, then $n=2m(m+4)+6=3y+6$ with $y=6b^2+16b+8$. Let $H=W(y-1,y+1;y)$ where $w\in V(H)$ is the unique vertex of degree $3$ with three neighbors of degree $3$. Next we will prove that $H$ is ETI. Assume that $N_H(w)=\{w_1,w_2,w_3\}$ with $\Tr(w)=h$ where at $w_1$ a pendent path of length $y-1$ is attached, at $w_2$ a pendent one of length $y+1$ is attached and $w_3$ lies on a triangle containing a $2$-degree vertex and another vertex at which a pendent path of length $y$ is attached. Similarly as above, we have $\Tr(H)-h=\bigcup_{i=0}^{3}D_i$ where $D_0=\{0,2,4,y,2y+3,3y+3\}$ and
\begin{align*}
D_1 & = \{(k+2)y+k(k+5)+3:\  k\in[y]\}, \\
D_2 & = \{ky+k(k+3)+2:\ k\in[y+1]\}, \\
D_3 & = \{ky+k(k+7)+4:\  k\in[y-1]\}.
\end{align*}
Then it follows that $\mathrm{ETr}(H)-2h=\bigcup_{i=0}^{3}E_i$ where $E_0=\{2,4,6,y,3y+3,4y+3,5y+6\}$ and \begin{align*}
E_1 & = \{(2k+3)y+2k^2+8k+2:\ k\in[y]\}, \\
E_2 & = \{(2k-1)y+2k^2+4k+2:\ k\in[y+1]\}, \\
E_3 & = \{(2k-1)y+2k^2+12k+2:\ k\in[y-1]\}.
\end{align*}
Since $3y+3$ and $4y+3$ are both odd with even $y\geq 8$, we have $E_0\cap (\bigcup_{i=1}^{3}E_i)=\emptyset$. To prove the ETI property of $H$, it suffices to prove that $\bigcup_{i=1}^{3}E_i$ is a partition set. For any two equal numbers $(2p-1)y+2p^2+4p+2\in E_2$ and $(2q-1)y+2q^2+12q+2\in E_3$, we have $(p-q)(y+p+q)=6q-2p$ with $p>q$. Since $y\geq 8$ is even, we conclude that $p-q\geq 2$ is even and $6q-2p\geq 2(y+p+q)$, that is, $3q-p\geq y+p+q$, which implies that $q>p$ as a clear contradiction. Thus $E_2\cap E_3=\emptyset$. If there are two equal numbers $(2p+3)y+2p^2+8p+2\in E_1$ and $(2q-1)y+2q^2+4q+2\in E_2$, then $2(y+q)=(q-p)(y+p+q+4)$ with $q > p$.
Since $y\geq 8$ is even, we conclude that $q - p\geq 2$ is even and $2(y + q) \ge 2(y + p + q + 4)$, which is impossible. Therefore $E_1\cap E_2 =\emptyset$. By a very similar reasoning to that for $E_1\cap E_2 =\emptyset$, we get $E_1\cap E_3 =\emptyset$ and omit its proof. Thus $H$ is ETI as desired.

If $m=3c$, then $n=3z+6$ with $z=2c(3c+4)$. For $c=1$, we assume that $L=C_3^*(1,20,24)$, which is of order $48$. From a routine check, we get that $L$ is ETI. For $c\geq 2$, we let $J=W(z-1,z+1;z)$. By an analogous proof we can show that $J$ is ETI, hence we omit the details.
\qed
\medskip

Let $A$ be a graph obtained by deleting an edge from the complete graph $K_4$. For four integers $p_2\geq p_1\geq0$ and $q_2\geq q_1\geq0$, we denote by $A(p_2,p_1;q_2,q_1)$ a graph obtained from $A$ by attaching two pendent paths of lengths $p_1$ and $p_2$, respectively, at two $3$-degree vertices of $A$, and two pendent paths of lengths $q_1$ and $q_2$, respectively, at two $2$-degree vertices of $A$. In particular, if $p_1=q_1=0$, then $A(p_2,p_1;q_2,q_1)$ will be written as $A(p_2;q_2)$. As an example, $A(3,2;2,1)$ is shown in Fig.~\ref{fig:A}.

\begin{figure}[ht!]
\begin{center}
\begin{tikzpicture}[scale=0.7,style=thick]
\tikzstyle{every node}=[draw=none,fill=none]
\def\vr{3pt} 

\begin{scope}[yshift = 0cm, xshift = 0cm]
\path (1,0) coordinate (x4);
\path (2,0) coordinate (x5);
\path (0.5,1) coordinate (y1);
\path (0.5,2) coordinate (y2);
\path (0,0) coordinate (x3);
\path (3,0) coordinate (x6);
\path (-1,0) coordinate (x2);
\path (-2,0) coordinate (x1);
\path (0.5,-1) coordinate (z2);
\path (1.5,-1) coordinate (z3);
\path (2.5,-1) coordinate (z4);
\path (4,0) coordinate (x7);

\draw (x3) -- (y1) -- (y2);
\draw (x1) -- (x2) -- (x3) -- (x4) -- (x5) --(x6) -- (x7);
\draw (x4) -- (y1);
\draw (x1) -- (z1);
\draw (z1) -- (x4) -- (z2) -- (z1);
\draw (z2) -- (z3) -- (z4);
\draw (x1)  [fill=black] circle (\vr);
\draw (x2)  [fill=black] circle (\vr);
\draw (y1)  [fill=black] circle (\vr);
\draw (y2)  [fill=black] circle (\vr);
\draw (z2)  [fill=black] circle (\vr);
\draw (z3)  [fill=black] circle (\vr);
\draw (z4)  [fill=black] circle (\vr);
\draw (x3)  [fill=black] circle (\vr);
\draw (x4)  [fill=black] circle (\vr);
\draw (x5)  [fill=black] circle (\vr);
\draw (x6)  [fill=black] circle (\vr);
\draw (x7)  [fill=black] circle (\vr);
\end{scope}
\end{tikzpicture}
\end{center}
\caption{The graph $A(3,2;2,1)$.}
\label{fig:A}
\end{figure}
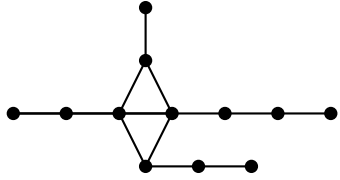

\begin{proposition}\label{c3} 
If $n\geq 10$ is even and $2n+20\in\mathcal{PS}$, then there is an ETI graph of order~$n$.
\end{proposition}

\proof By Lemma \ref{equi}, we can assume that $n=2m(m+6)+8$ with $m \in \mathbb{N}$. We first assume that $m$ is even. Let $H=A(m(m+6)+2;m(m+6)+2)$ with $v\in V(H)$ as the unique vertex of degree $4$ in $H$ with $\Tr(v)=h$. Then $\Tr(H)-h=A_0\cup A_1\cup A_2$ with $A_0=\{m(m+6)+2,2m(m+6)+5\}$, $A_1=\{k^2+3k:\  k\in[m(m+6)+2]_0\}$ and $A_2=A_1+1$. Then $\mathrm{ETr}(H)-2h=B_0\cup B_1\cup B_2$ with
\begin{align*}
B_0 & = \{1,m(m+6)+2,2m(m+6)+5,m(m+6)+3,3m(m+6)+7\}, \\
B_1 & = \{2k^2+4k-2:\ k\in[m(m+6)+2]\},~ B_2=B_1+2.
\end{align*}
Note that the minimum difference of two distinct integers in $B_1$ is equal to \linebreak $\min_{k\in[m(m+6)+1]}{2(2k+3)}=10$. So $B_1\cap B_2=\emptyset$. Since $m$ is even, $B_0$ contains only one even integer $m(m+6)+2$. Next we show that $m(m+6)+2\notin B_1\cup B_2$. Otherwise, we first assume that $m(m+6)+2=2k^2+4k-2$ with some $k\in [m(m+6)+2]$. It follows that $m(m + 6) + 6 = 2(k + 1)^2$, which is impossible because the left-hand side is congruent to $6$ modulo $8$ (since $m$ is even), while the right-hand side is congruent either to $0$ or to $2$ modulo~$8$. If there is some $k\in [m(m+6)+2]$ with $m(m+6)+2=2k^2+4k$, then similarly to above, we can get a contradiction. Therefore $H$ is ETI.

Now we assume that $m$ is odd. Let $L=A(m(m+6)+1,1;m(m+6)+1,1)$ with $u\in V(L)$ of degree $4$, at which a pendent path of length $m(m+6)+1$ is attached, with $\Tr_L(u)=\ell$. Then $\Tr(L)-\ell=D_0\cup D_1\cup D_2$ with $D_0=\{m(m+6),2m(m+6)+2,3m(m+6)+6,4m(m+6)+8\}$, $D_1=\{k^2+5k:\ k\in[m(m+6)+1]_0\}$ and $D_2=D_1+2$. It follows that $\mathrm{ETr}(L)-2\ell=E_0\cup E_1\cup E_2$ with $$E_0=\{2,m(m+6),m(m+6)+2,2m(m+6)+2,3m(m+6)+2,4m(m+6)+6,6m(m+6)+10\},$$ $E_1=\{2k^2+8k-4:\ k\in[m(m+6)+1]\}$ and $E_2=E_1+4$. Since the difference of any two distinct numbers in $E_1$ is at least $\min_{k\in[m(m+6)]} 4k+10=14>4$, we have $E_1\cap E_2=\emptyset$. Note that $E_0$ contains only four even integers $2$, $2m(m+6)+2$, $4m(m+6)+6$ and $6m(m+6)+10$. To prove the ETI property of graph $L$, it suffices to show that $$(E_1\cup E_2)\cap \{2m(m+6)+2, 4m(m+6)+6, 6m(m+6)+10\}=\emptyset.$$

If $2m(m+6)+2=2k^2+8k-4$ for some $k\in [m(m+6)+1]$, then $k(k+4)=m(m+6)+3$ with $k>m$, that is, $(k-m)(k+m+4)=2m+3$. It follows that $2m+3\geq k+m+4>2m+4$ as a clear contradiction. Thus $2m(m+6)+2\notin E_1$. By a similar reasoning, we have $2m(m+6)+2\notin E_2$.

If $4m(m+6)+6=2k^2+8k$ for some $k\in [m(m+6)+1]$, then $k(k+4)=2m(m+6)+3$, that is, $(k + 2)^2 + 11 = 2(m + 3)^2$. Since $m$ is odd, the right-hand side is divisible by $8$, which yields a contradiction because $5$ is not a quadratic residue modulo $8$. Therefore $4m(m+6)+6\notin E_2$. Similarly to above, we get $4m(m+6)+6\notin E_1$.

Finally we consider the case of $6m(m+6)+10$. If $6m(m+6)+10=2k^2+8k-4$ for some $k\in [m(m+6)+1]$, we have $3m(m+6)+7=k(k+4)$, that is, $3(m + 3)^2 = (k + 2)^2 + 16$. The left-hand side is divisible by $3$, which yields a contradiction because $2$ is not a quadratic residue modulo $3$. Thus $6m(m+6)+10\notin E_1$. If $6m(m+6)+10=2k^2+8k$ for some $k\in [m(m+6)+1]$, then we get $3(m + 3)^2 = (k + 2)^2 + 18$. Since $m$ is odd, it follows that the left-hand side is divisible by $4$, which yields a contradiction because $2$ is not a quadratic residue modulo $4$. Therefore $6m(m+6)+10\notin E_2$, and $J$ is ETI as desired. 
\qed
\medskip

\noindent
From Propositions \ref{H-n} and \ref{c1}--\ref{c3}, we get the following result.

\begin{theorem}\label{even}
For any even $n \geq 10$, there exists a chemical ETI graph of order $n$. 
\end{theorem}

\section{Subcubic trees that are both TI and ETI}
\label{sec:main}

In the present section, we prove Theorem~\ref{th:main} by relying on the family of subcubic trees that we call \emph{saw graphs}. As it turns out, these graphs are both TI and ETI under suitable conditions. We mention in passing that saw graphs were discovered through computational experiments using a software library that is currently still under development.

We now introduce the terminology and define the family formally. For $p \in \mathbb{N}_0$ and $n \ge 4p + 5$, we define the \emph{saw graph} $\mathrm{SW}_{n, p}$ as follows. If $n$ is odd, $\mathrm{SW}_{n, p}$ is obtained from the path $P_{n - 2p - 1}$ by selecting one of its two central vertices and attaching a leaf to each vertex within distance $p$ from it; see Fig.~\ref{odd_saw_fig}. If $n$ is even, $\mathrm{SW}_{n, p}$ is obtained from the path $P_{n - 2p - 2}$ by selecting one of its two central vertices, attaching a pendent path of length two to it, and attaching a leaf to each vertex at distance between $1$ and $p$ from it; see Fig.~\ref{even_saw_fig}.

The following two propositions give sufficient conditions under which a saw graph of odd order is TI and ETI, respectively.

\begin{proposition}\label{odd_saw_ti}
If $p \in \mathbb{N}_0$, $n \ge \max \{ 4p + 5, p(2p + 1) + 3 \}$ is odd, and 
    \begin{equation}\label{ti_odd_cond}
    \mathcal{PS} \cap \left( n + \{ 2p^2 + p - 2, 2p^2 + 3p - 1 \} + \{ j(2j - 1), j(2j + 1):\ j \in [p]_0\} \right) = \emptyset, 
    \end{equation}
    then $\mathrm{SW}_{n, p}$ is TI.
    
\end{proposition}
\begin{proof}
    Let the vertices of $\mathrm{SW}_{n, p}$ be labeled as in Fig.~\ref{odd_saw_fig} and let $\beta \coloneqq \Tr(x_0)$. The repeated use of Lemmas~\ref{equal} and \ref{pend} yields 
\begin{align*}
\Tr(x_j) & = \beta + j(2j - 1),\ j \in [p]_0\,,\\
\Tr(x_j') & = \beta + j(2j - 1) + (n - 2),\ j \in [p]_0\,,\\
\Tr(y_j) & = \beta + p(2p - 1) + j(j + 4p),\ j \in [(n - 4p - 1)/2]_0\,,\\
\Tr(z_j) & = \beta + j(2j + 1),\ j \in [p]_0\,,\\
\Tr(z_j') & = \beta + j(2j + 1) + (n - 2),\ j \in [p]_0\,, \text{ and}\\
\Tr(w_j) & = \beta + p(2p + 1) + j(j + 4p + 2),\ j \in [(n - 4p - 3)/2]_0\,.    
\end{align*}    
Note that $x_0 \equiv z_0$ uniquely attains the smallest transmission.

    First, we prove that the $x_j$, $x_j'$, $z_j$ and $z_j'$ vertices have mutually distinct transmissions. By way of contradiction, suppose that $\Tr(x_j) = \Tr(z_h)$ for some $j \in [p]$ and $h \in [p]$. Then $j(2j - 1) = h(2h + 1)$, hence $(4j - 1)^2 = (4h + 1)^2$. If $4h + 1 = 4j - 1$, then $4j - 4h = 2$, which is impossible, and if $4h + 1 = 1 - 4j$, then $h = -j$, which is also impossible. Therefore, the $x_j$ and $z_j$ vertices have mutually distinct transmissions. Since $\Tr(x_j') = \Tr(x_j) + (n - 2)$ and $\Tr(z_j') = \Tr(z_j) + (n - 2)$, this implies that the $x_j'$ and $z_j'$ vertices also have mutually distinct transmissions. It remains to observe that the transmissions of the $x_j$ and $z_j$ vertices are at most $\beta + p(2p + 1)$, while those of the $x_j'$ and $z_j'$ vertices are at least $\beta + (n - 2)$, so the two ranges are disjoint because $n \ge p(2p + 1) + 3$.

    Now, by way of contradiction, suppose that $\Tr(y_j) = \Tr(w_h)$ for some $j \in [\frac{n - 4p - 1}{2}]$ and $h \in [\frac{n - 4p - 3}{2}]$. Then $p(2p - 1) + j(j + 4p) = p(2p + 1) + h(h + 4p + 2)$, hence $(h + 2p + 1)^2 - (j + 2p)^2 = 2p + 1$. Since $j + 2p \ge 2p + 1$, any perfect square above $(j + 2p)^2$ is greater than $(j + 2p)^2$ by at least $(2p + 2)^2 - (2p + 1)^2 = 4p + 3$, yielding a contradiction. Therefore, the $y_j$ and $w_j$ vertices, excluding $y_0$ and $w_0$, have mutually distinct transmissions. Also, the transmissions of these vertices are at least $\beta + p(2p - 1) + (4p + 1)$, while those of the $x_j$ and $z_j$ vertices are at most $\beta + p(2p + 1)$, so the two ranges are disjoint. Thus, it remains to show that a $y_j$ or a $w_j$ vertex does not have the same transmission as a $x_j'$ or a $z_j'$ vertex.

Recall that none of the numbers in \eqref{ti_odd_cond} is a perfect square. If $\Tr(y_j) = \Tr(x_h')$, then 
$$p(2p - 1) + j(j + 4p) = h(2h - 1) + (n - 2),$$ 
hence $(j + 2p)^2 = n + (2p^2 + p - 2) + h(2h - 1)$, a contradiction. If $\Tr(y_j) = \Tr(z_h')$, then 
$$p(2p - 1) + j(j + 4p) = h(2h + 1) + (n - 2),$$ 
hence $(j + 2p)^2 = n + (2p^2 + p - 2) + h(2h + 1)$, a contradiction. If $\Tr(w_j) = \Tr(x_h')$, then 
$$p(2p + 1) + j(j + 4p + 2) = h(2h - 1) + (n - 2),$$ 
hence $(j + 2p + 1)^2 = n + (2p^2 + 3p - 1) + h(2h - 1)$, a contradiction. Finally, if $\Tr(w_j) = \Tr(z_h')$, then 
$$p(2p + 1) + j(j + 4p + 2) = h(2h + 1) + (n - 2),$$ 
hence $(j + 2p + 1)^2 = n + (2p^2 + 3p - 1) + h(2h + 1)$, yielding a contradiction.
\end{proof}

\begin{figure}[t]
\centering
\begin{tikzpicture}[scale=1.0]
\tikzstyle{vertex}=[draw,circle,minimum size=4pt,inner sep=1pt,fill=black]
\tikzstyle{edge}=[draw,thick]
\tikzstyle{dedge}=[draw,thick,dashed]

\node[vertex, label=below:$z_0 \equiv x_0$] (0) at (0, 0) {};
\node[vertex, label=below:$x_1$] (1) at (1, 0) {};
\node[vertex, label=below:$x_2$] (2) at (2, 0) {};
\node[vertex, label=below:$x_p \equiv y_0$] (3) at (3, 0) {};
\node[vertex, label=below:$z_1$] (4) at (-1, 0) {};
\node[vertex, label=below:$z_2$] (5) at (-2, 0) {};
\node[vertex, label=below:$w_0 \equiv z_p$] (6) at (-3, 0) {};
\node[vertex, label=above:$z_0' \equiv x_0'$] (7) at (0, 1) {};
\node[vertex, label=above:$x_1'$] (8) at (1, 1) {};
\node[vertex, label=above:$x_2'$] (9) at (2, 1) {};
\node[vertex, label=above:$x_p'$] (10) at (3, 1) {};
\node[vertex, label=above:$z_1'$] (11) at (-1, 1) {};
\node[vertex, label=above:$z_2'$] (12) at (-2, 1) {};
\node[vertex, label=above:$z_p'$] (13) at (-3, 1) {};

\node[vertex, label=below:$y_1$] (14) at (4, 0) {};
\node[vertex, label=below:$y_2$] (15) at (5, 0) {};
\node[vertex, label=below:$y_\frac{n - 4p - 1}{2}$] (16) at (6, 0) {};
\node[vertex, label=below:$w_1$] (17) at (-4, 0) {};
\node[vertex, label=below:$w_2$] (18) at (-5, 0) {};
\node[vertex, label=below:$w_\frac{n - 4p - 3}{2}$] (19) at (-6, 0) {};

\path[edge] (0) -- (1);
\path[edge] (1) -- (2);
\path[dedge] (2) -- (3);
\path[edge] (0) -- (4);
\path[edge] (4) -- (5);
\path[dedge] (5) -- (6);
\path[edge] (0) -- (7);
\path[edge] (1) -- (8);
\path[edge] (2) -- (9);
\path[edge] (3) -- (10);
\path[edge] (4) -- (11);
\path[edge] (5) -- (12);
\path[edge] (6) -- (13);

\path[edge] (3) -- (14);
\path[edge] (14) -- (15);
\path[dedge] (15) -- (16);
\path[edge] (6) -- (17);
\path[edge] (17) -- (18);
\path[dedge] (18) -- (19);

\end{tikzpicture}
\caption{The saw graph $\mathrm{SW}_{n, p}$ for $p \in \mathbb{N}_0$ and odd $n \ge 4p + 5$.}
\label{odd_saw_fig}
\end{figure}
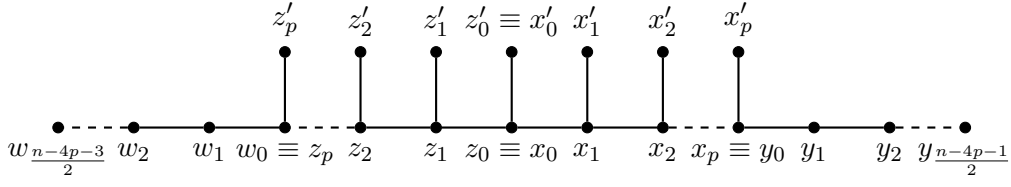

\begin{figure}[t]
\centering
\begin{tikzpicture}[scale=1.0]
\tikzstyle{vertex}=[draw,circle,minimum size=4pt,inner sep=1pt,fill=black]
\tikzstyle{edge}=[draw,thick]
\tikzstyle{dedge}=[draw,thick,dashed]

\node[vertex, label=above:$z_0 \equiv x_0$] (0) at (0, 0) {};
\node[vertex, label=below:$x_1$] (1) at (1, 0) {};
\node[vertex, label=below:$x_2$] (2) at (2, 0) {};
\node[vertex, label=below:$x_p \equiv y_0$] (3) at (3, 0) {};
\node[vertex, label=below:$z_1$] (4) at (-1, 0) {};
\node[vertex, label=below:$z_2$] (5) at (-2, 0) {};
\node[vertex, label=below:$w_0 \equiv z_p$] (6) at (-3, 0) {};
\node[vertex, label=right:$z_0' \equiv x_0'$] (7) at (0, -1) {};
\node[vertex, label=above:$x_1'$] (8) at (1, 1) {};
\node[vertex, label=above:$x_2'$] (9) at (2, 1) {};
\node[vertex, label=above:$x_p'$] (10) at (3, 1) {};
\node[vertex, label=above:$z_1'$] (11) at (-1, 1) {};
\node[vertex, label=above:$z_2'$] (12) at (-2, 1) {};
\node[vertex, label=above:$z_p'$] (13) at (-3, 1) {};

\node[vertex, label=below:$y_1$] (14) at (4, 0) {};
\node[vertex, label=below:$y_2$] (15) at (5, 0) {};
\node[vertex, label=below:$y_\frac{n - 4p - 2}{2}$] (16) at (6, 0) {};
\node[vertex, label=below:$w_1$] (17) at (-4, 0) {};
\node[vertex, label=below:$w_2$] (18) at (-5, 0) {};
\node[vertex, label=below:$w_\frac{n - 4p - 4}{2}$] (19) at (-6, 0) {};

\node[vertex, label=right:$z_0'' \equiv x_0''$] (20) at (0, -2) {};

\path[edge] (0) -- (1);
\path[edge] (1) -- (2);
\path[dedge] (2) -- (3);
\path[edge] (0) -- (4);
\path[edge] (4) -- (5);
\path[dedge] (5) -- (6);
\path[edge] (0) -- (7);
\path[edge] (1) -- (8);
\path[edge] (2) -- (9);
\path[edge] (3) -- (10);
\path[edge] (4) -- (11);
\path[edge] (5) -- (12);
\path[edge] (6) -- (13);

\path[edge] (3) -- (14);
\path[edge] (14) -- (15);
\path[dedge] (15) -- (16);
\path[edge] (6) -- (17);
\path[edge] (17) -- (18);
\path[dedge] (18) -- (19);

\path[edge] (7) -- (20);

\end{tikzpicture}
\caption{The saw graph $\mathrm{SW}_{n, p}$ for $p \in \mathbb{N}_0$ and even $n \ge 4p + 6$.}
\label{even_saw_fig}
\end{figure}
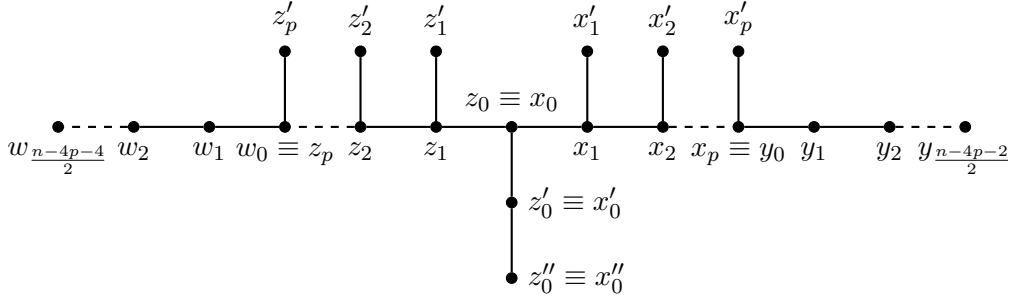

\begin{proposition}\label{odd_saw_eti}
If $p \in \mathbb{N}_0$, $n \ge \max \{ 4p + 5, 4p^2 - 2p + 5 \}$ is odd, and 
    \begin{equation}\label{eti_odd_cond}
       \mathcal{PS} \cap \left(  2n + \{ 8p^2 + 4p - 5, 8p^2 + 12p - 1 \} + \{ 4j(2j - 1), 4j(2j + 1):\ j \in [p]_0 \}\right) = \emptyset,
    \end{equation}
    then $\mathrm{SW}_{n, p}$ is ETI.
\end{proposition}
\begin{proof}
    Let the vertices of $\mathrm{SW}_{n, p}$ be labeled as in Fig.~\ref{odd_saw_fig} and let $\beta \coloneqq \Tr(x_0)$. From the proof of Proposition \ref{odd_saw_ti}, it follows that 
\begin{align*}
\Tr(x_{j - 1} x_j) & = 2 \beta + (4j^2 - 6j + 3),\ j \in [p]\,,\\
\Tr(x_j x_j') & = 2 \beta + 2j(2j - 1) + (n - 2),\ j \in [p]_0\,,\\
\Tr(y_{j - 1} y_j) & = 2 \beta + 2p(2p - 1) + (2j^2 + j(8p - 2) - (4p - 1)),\ j \in [(n - 4p - 1)/2]\,,\\
\Tr(z_{j - 1} z_j) & = 2 \beta + (4j^2 - 2j + 1),\ j \in [p]\,,\\
\Tr(z_j z_j') & = 2 \beta + 2j(2j + 1) + (n - 2),\ j \in [p]_0\,, \text{ and}\\
\Tr(w_{j - 1} w_j) & = 2 \beta + 2p(2p + 1) + (2j^2 + j(8p + 2) - (4p + 1))\,, j \in [(n - 4p - 3)/2]\,.
\end{align*}
First, we prove that the $x_{j - 1} x_j$, $x_j x_j'$, $z_{j - 1} z_j$ and $z_j z_j'$ edges have mutually distinct transmissions. By way of contradiction, suppose that $\Tr(x_{j - 1} x_j) = \Tr(z_{h - 1} z_h)$ for some $j \in [p]$ and $h \in [p]$. Then $4j^2 - 6j + 3 = 4h^2 - 2h + 1$, hence $(4j - 3)^2 = (4h - 1)^2$. If $4h - 1 = 4j - 3$, then $4j - 4h = 2$, which is impossible, and if $4h - 1 = 3 - 4j$, then $j + h = 1$, which is also impossible. Therefore, the $x_{j - 1} x_j$ and $z_{j - 1} z_j$ edges have mutually distinct transmissions. From the proof of Proposition \ref{odd_saw_ti}, the $x_j$ and $z_j$ vertices have mutually distinct transmissions. Since $\Tr(x_j x_j') = 2 \Tr(x_j) + (n - 2)$ and $\Tr(z_j z_j') = 2 \Tr(z_j) + (n - 2)$, it follows that the $x_j x_j'$ and $z_j z_j'$ edges also have mutually distinct transmissions. Note that if $p = 0$, then there are no $x_{j - 1} x_j$ or $z_{j - 1} z_j$ edges. It remains to observe that, if $p \ge 1$, then the transmissions of the $x_{j - 1} x_j$ and $z_{j - 1} z_j$ edges are at most $2 \beta + (4p^2 - 2p + 1)$, while those of the $x_j x_j'$ and $z_j z_j'$ edges are at least $2 \beta + (n - 2)$, so the two ranges are disjoint because $n \ge 4p^2 - 2p + 5$.

    Now, by way of contradiction, suppose that $\Tr(y_{j - 1} y_j) = \Tr(w_{h - 1} w_h)$ for some $j \in [\frac{n - 4p - 1}{2}]$ and $h \in [\frac{n - 4p - 3}{2}]$. Then $2j^2 + j(8p - 2) - (6p - 1) = 2h^2 + h(8p + 2) - (2p + 1)$, hence $(2h + 4p + 1)^2 - (2j + 4p - 1)^2 = 8p + 4$. If $j \ge 2$, then $2j + 4p - 1 \ge 4p + 3$, so any perfect square above $(2j + 4p - 1)^2$ is greater than $(2j + 4p - 1)^2$ by at least $(4p + 4)^2 - (4p + 3)^2 = 8p + 7$, yielding a contradiction. If $j = 1$, then $(2h + 4p + 1)^2 = (4p + 2)^2 + 1$, which is impossible. Therefore, the $y_{j - 1} y_j$ and $w_{j - 1} w_j$ edges have mutually distinct transmissions. Also, provided $p \ge 1$, the transmissions of these edges are at least $2 \beta + 2p(2p - 1) + (4p + 1)$, while those of the $x_{j - 1} x_j$ and $z_{j - 1} z_j$ edges are at most $2 \beta + (4p^2 - 2p + 1)$, so the two ranges are disjoint. Thus, it remains to show that a $y_{j - 1} y_j$ or $w_{j - 1} w_j$ edge does not have the same transmission as a $x_j x_j'$ or $z_j z_j'$ edge.

    Recall that none of the numbers in \eqref{eti_odd_cond} are a perfect square. If $\Tr(y_{j - 1} y_j) = \Tr(x_h x_h')$, then
    \[
        2p(2p - 1) + (2j^2 + j(8p - 2) - (4p - 1)) = 2h(2h - 1) + (n - 2),
    \]
    hence $(2j + 4p - 1)^2 = 2n + (8p^2 + 4p - 5) + 4h(2h - 1)$, yielding a contradiction. If $\Tr(y_{j - 1} y_j) = \Tr(z_h z_h')$, then
    \[
        2p(2p - 1) + (2j^2 + j(8p - 2) - (4p - 1)) = 2h(2h + 1) + (n - 2),
    \]
    hence $(2j + 4p - 1)^2 = 2n + (8p^2 + 4p - 5) + 4h(2h + 1)$, yielding a contradiction. If $\Tr(w_{j - 1} w_j) = \Tr(x_h x_h')$, then
    \[
        2p(2p + 1) + (2j^2 + j(8p + 2) - (4p + 1)) = 2h(2h - 1) + (n - 2),
    \]
    hence $(2j + 4p + 1)^2 = 2n + (8p^2 + 12p - 1) + 4h(2h - 1)$, yielding a contradiction. If $\Tr(w_{j - 1} w_j) = \Tr(z_h z_h')$, then
    \[
        2p(2p + 1) + (2j^2 + j(8p + 2) - (4p + 1)) = 2h(2h + 1) + (n - 2),
    \]
    hence $(2j + 4p + 1)^2 = 2n + (8p^2 + 12p - 1) + 4h(2h + 1)$, yielding a contradiction.
\end{proof}

We now establish analogous sufficient conditions for saw graphs of even order to be TI and ETI. Since the proofs follow the same approach as those of Propositions~\ref{odd_saw_ti} and~\ref{odd_saw_eti}, we omit the details.

\begin{proposition}\label{even_saw_ti}
If $p \in \mathbb{N}_0$, $n \ge 2p(p + 1) + 6$ is even, and
\begin{align*}
& \mathcal{PS} \cap \left( 4n + \{ 8p^2 + 8p - 7, 8p^2 + 16p + 1 \} + \{ -8 \} \cup \{ 8j^2, 8j(j + 1):\ j \in [p]\}\right) = \emptyset, \\
& \mathcal{PS} \cap \left( 8n + \{ 8p^2 + 8p - 23, 8p^2 + 16p - 15 \}\right) = \emptyset,  
\end{align*}
then $\mathrm{SW}_{n, p}$ is TI.
\end{proposition}

\begin{proposition}\label{even_saw_eti}
If $p \in \mathbb{N}_0$, $n \ge 4p^2 + 6$ is even, and 
\begin{align*}
& \mathcal{PS} \cap \left( 2n + \{ 8p^2 + 8p - 4, 8p^2 + 16p + 4 \} + \{ -4 \} \cup \{ 8j^2, 8j(j + 1):\ j \in [p]\} \right), \\
& \mathcal{PS} \cap \left( 6n + \{ 8p^2 + 8p - 20, 8p^2 + 16p - 12 \}\right),
\end{align*}
then $\mathrm{SW}_{n, p}$ is ETI.
\end{proposition}

\noindent
The following corollaries are immediate consequences of Propositions~\ref{odd_saw_ti} and~\ref{odd_saw_eti}.

\begin{corollary}\label{odd_saw_0}
If $n \ge 5$ is odd, $\mathcal{PS} \cap  [n - 2, n - 1] = \emptyset$, and there are no odd perfect squares in $[2n - 5, 2n - 1]$, then $\mathrm{SW}_{n, 0}$ is both TI and ETI.
\end{corollary}

\begin{corollary}\label{odd_saw_1}
If $n \ge 9$ is odd, $\mathcal{PS} \cap [n + 1, n + 7] = \emptyset$,  and there are no odd perfect squares in $[2n + 7, 2n + 31]$, then $\mathrm{SW}_{n, 1}$ is both TI and ETI.
\end{corollary}

\begin{corollary}\label{odd_saw_2}
If $n \ge 17$ is odd, $\mathcal{PS} \cap [n + 8, n + 23] = \emptyset$, and there are no odd perfect squares in $[2n + 35, 2n + 95]$, then $\mathrm{SW}_{n, 2}$ is both TI and ETI.
\end{corollary}

\noindent
The following corollaries are obtained from Propositions~\ref{even_saw_ti} and~\ref{even_saw_eti}.

\begin{corollary}\label{even_saw_0}
    Let $n \ge 6$ be even and assume that there are no even perfect squares in $[2n - 8, \linebreak 2n]$, no odd perfect squares in $[4n - 15, 4n - 7]$, no even perfect squares in $[6n - 20, 6n - 12]$, and no odd perfect squares in $[8n - 23, 8n - 15]$. Then $\mathrm{SW}_{n, 0}$ is both TI and ETI.
\end{corollary}

\begin{corollary}\label{even_saw_1}
    Let even $n \ge 10$ and assume that there are no even perfect squares in $[2n + 8, 2n + 44]$, no odd perfect squares in $[4n + 1, 4n + 41]$, no even perfect squares in $[6n - 4, 6n + 12]$, and no odd perfect squares in $[8n - 7, 8n + 9]$. Then $\mathrm{SW}_{n, 1}$ is both TI and ETI.
\end{corollary}

\begin{corollary}\label{even_saw_3}
    Let $n \ge 42$ be even and assume that there are no even perfect squares in $[2n + 88, 2n + 220]$, no odd perfect squares in $[4n + 81, 4n + 217]$, no even perfect squares in $[6n + 76, 6n + 108]$, and no odd perfect squares in $[8n + 73, 8n + 105]$. Then $\mathrm{SW}_{n, 3}$ is both TI and ETI.
\end{corollary}

\begin{corollary}\label{even_saw_5}
    Let $n \ge 106$ be even and assume that there are no even perfect squares in $[2n + 232, 2n + 524]$, no odd perfect squares in $[4n + 225, 4n + 521]$, no even perfect squares in $[6n + 220, 6n + 268]$, and no odd perfect squares in $[8n + 217, 8n + 265]$. Then $\mathrm{SW}_{n, 5}$ is both TI and ETI.
\end{corollary}

\begin{corollary}\label{even_saw_8}
    Let $n \ge 262$ be even and assume that there are no even perfect squares in $[2n + 568, 2n + 1220]$, no odd perfect squares in $[4n + 561, 4n + 1217]$, no even perfect squares in $[6n + 556, 6n + 628]$, and no odd perfect squares in $[8n + 553, 8n + 625]$. Then $\mathrm{SW}_{n, 8}$ is both TI and ETI.
\end{corollary}

As the next step, we show that the saw graphs yield a subcubic tree of order $n \in \mathbb{N}$ that is both TI and ETI, provided $n$ is sufficiently large. This is done separately for the odd and even orders, as follows.

\begin{lemma}\label{odd_big_lemma}
    If $n \ge 317$ is odd, then at least one of $\mathrm{SW}_{n, 0}$, $\mathrm{SW}_{n, 1}$ and $\mathrm{SW}_{n, 2}$ is both TI and ETI.
\end{lemma}

\begin{proof}
    By way of contradiction, suppose that none of $\mathrm{SW}_{n, 0}$, $\mathrm{SW}_{n, 1}$ and $\mathrm{SW}_{n, 2}$ are both TI and ETI. Then, by Corollaries \ref{odd_saw_0}--\ref{odd_saw_2}, there are at least two perfect squares in $[n - 2, n + 23]$, or at least two odd perfect squares in $[2n - 5, 2n + 95]$. In the former case, the difference between the two perfect squares is at most $25$, hence $n + 23 < 14^2$, which gives $n < 173$, yielding a contradiction. In the latter case, the difference between the two odd perfect squares is at most $100$, hence $2n + 95 < 27^2$, which gives $n < 317$, yielding a contradiction.
\end{proof}

\begin{lemma}\label{even_big_lemma}
    If $n \ge 47440$ is even, then at least one of $\mathrm{SW}_{n, 0}$, $\mathrm{SW}_{n, 1}$, $\mathrm{SW}_{n, 3}$, $\mathrm{SW}_{n, 5}$ and $\mathrm{SW}_{n, 8}$ is both TI and ETI.
\end{lemma}
\begin{proof}
    By way of contradiction, suppose that none of $\mathrm{SW}_{n, 0}$, $\mathrm{SW}_{n, 1}$, $\mathrm{SW}_{n, 3}$, $\mathrm{SW}_{n, 5}$ and $\mathrm{SW}_{n, 8}$ are both TI and ETI. Then, by Corollaries \ref{even_saw_0}--\ref{even_saw_8}, there are at least two even perfect squares in $[2n - 8, 2n + 1220]$, or at least two odd perfect squares in $[4n - 15, 4n + 1217]$, or at least two even perfect squares in $[6n - 20, 6n + 628]$, or at least two odd perfect squares in $[8n - 23, 8n + 625]$. We split the argument into four cases corresponding to the four assertions.
    
    \bigskip\noindent
    \textbf{Case 1}: there are at least two even perfect squares in $[2n - 8, 2n + 1220]$. \\
    In this case, the difference between the two even perfect squares is at most $1228$, hence $2n + 1220 < 310^2$, which gives $n < 47440$, yielding a contradiction.

    \bigskip\noindent
    \textbf{Case 2}: there are at least two odd perfect squares in $[4n - 15, 4n + 1217]$.\\
    In this case, the difference between the two odd perfect squares is at most $1232$, hence $4n + 1217 < 311^2$, which gives $n < 23876$, yielding a contradiction.

    \bigskip\noindent
    \textbf{Case 3}: there are at least two even perfect squares in $[6n - 20, 6n + 628]$.\\
    In this case, the difference between the two even perfect squares is at most $648$, hence $6n + 628 < 164^2$, which gives $n < 4378$, yielding a contradiction.

    \bigskip\noindent
    \textbf{Case 4}: there are at least two odd perfect squares in $[8n - 23, 8n + 625]$.\\
    In this case, the difference between the two odd perfect squares is at most $648$, hence $8n + 625 < 165^2$, which gives $n < 3325$, yielding a contradiction.
\end{proof}

\noindent
We are now in a position to finalize the proof of Theorem \ref{th:main}.

\begin{proof}[Proof of Theorem \ref{th:main}]
    First, assume that $n$ is odd. It is straightforward to verify, either manually or by computer, that the only nontrivial TI tree of odd order at most $7$ is $S(1, 2, 3) \cong \mathrm{SW}_{7, 0}$. Since this tree is not ETI, the nonexistence part of the theorem for odd orders immediately follows. The existence of a subcubic tree of any odd order $n \ge 317$ that is both TI and ETI follows from Lemma \ref{odd_big_lemma}.

    Now, assume that $n \in \{ 9, 11, 13, \ldots, 315 \}$. By executing the \texttt{odd\_resolver.py} \texttt{SageMath} script from~\cite{GitHub}, we conclude that at least one of $\mathrm{SW}_{n, 0}$, $\mathrm{SW}_{n, 1}$ and $\mathrm{SW}_{n, 2}$ is defined and both TI and ETI, provided $n \notin \{ 11, 13, 15, 25, 37, 51, 63 \}$. The script also confirms that the remaining odd orders can be settled using the constructions listed in Table \ref{singular_constructions}.

    We now consider the case where $n$ is even. Using the TI tree generation algorithm from~\cite{StoDam2026}, one can verify that the only TI tree of even order at most $14$ is the tree shown in Fig.~\ref{ti_14_fig}. Although this graph is ETI, it is not subcubic. This completes the nonexistence part of the theorem for even orders. The existence of a subcubic tree of any even order $n \ge 47440$ that is both TI and ETI follows from Lemma~\ref{even_big_lemma}.

    Now, assume that $n \in \{ 16, 18, 20, \ldots, 47438 \}$. Executing the \texttt{even\_resolver.py} \texttt{SageMath} script from \cite{GitHub} shows that at least one of $\mathrm{SW}_{n, p}$ for $p \in \{ 0, 1, 2, \ldots, 8 \}$ is defined and is both TI and ETI, provided
    \begin{align*}
        n \notin \{&16, 18, 20, 22, 24, 26, 30, 32, 34, 36, 38, 44, 48, 50, 54, 56, 58, 60, 68,\\
        &72, 76, 92, 102, 114, 116, 122, 132, 158, 200, 278, 296, 342, 420, 466 \} .
    \end{align*}
    The script also confirms that the remaining even orders can be settled using the constructions listed in Table \ref{singular_constructions}.
\end{proof}

\begin{remark}
    The constructions listed in Table \ref{singular_constructions} were found using a combination of different approaches. Some graphs were identified by filtering out the ETI trees among the TI trees generated by the algorithm from \cite{StoDam2026}, while others were constructed manually. The remaining cases, particularly those involving larger values of $n$, were discovered via a genetic algorithm whose source code is available in \cite{GitHub}.
\end{remark}

\begin{remark}
    There is a unique subcubic tree of order $n$ that is both TI and ETI, for each $n \in \{ 9, 11, 13, 16, 18 \}$.
\end{remark}

\begin{table}[ht]
\centering
\begin{tabular}{r|c||r|c}
$n$ & Construction & $n$ & Construction \\
\hline \hline
$11$ & $C_9(3, 5)$ & $56$ & $C_{51}(5, 1; 24, 1; 26, 2; 30, 1)$ \\
$13$ & $C_9(3, 1; 5, 3)$ & $58$ & $C_{49}(24, 2; 26, 2; 28, 2; 43, 3)$ \\
$15$ & $C_{13}(5, 7)$ & $60$ & $C_{53}(15, 2; 25, 1; 27, 2; 35, 1; 36, 1)$ \\
$16$ & $C_{10}(3, 1; 5, 4; 7, 1)$ & $63$ & $C_{61}(29, 31)$ \\
$18$ & $C_{10}(3, 1; 5, 4; 6, 3)$ & $68$ & $C_{58}(8, 2; 22, 3; 29, 2; 35, 3)$ \\
$20$ & $C_{17}(5, 1; 9, 2)$ & $72$ & $C_{63}(11, 1; 15, 1; 29, 1; 33, 2; 37, 1; 47, 3)$ \\
$22$ & $C_{18}(7, 1; 9, 2; 11, 1)$ & $76$ & $C_{71}(11, 1; 36, 2; 39, 1; 68, 1)$ \\
$24$ & $C_{18}(4, 1; 5, 3; 8, 2)$ & $92$ & $C_{88}(43, 1; 45, 2; 50, 1)$ \\
$25$ & $C_{22}(9, 11, 14)$ & $102$ & $C_{92}(11, 3; 19, 2; 43, 1; 45, 2; 64, 1; 66, 1)$ \\
$26$ & $C_{18}(4, 2; 5, 2; 9, 2; 13, 2)$ & $114$ & $C_{109}(54, 1; 55, 2; 56, 2)$ \\
$30$ & $C_{22}(3, 1; 4, 2; 12, 2; 17, 3)$ & $116$ & $C_{113}(55, 1; 57, 2)$ \\
$32$ & $C_{25}(7, 1; 13, 3; 14, 2; 22, 1)$ & $122$ & $C_{113}(9, 1; 44, 1; 52, 1; 57, 2; 58, 1; 63, 1; 98, 2)$ \\
$34$ & $C_{29}(5, 2; 14, 2; 15, 1)$ & $132$ & $C_{129}(58, 1; 65, 2)$ \\ 
$36$ & $C_{31}(8, 1; 16, 2; 17, 1; 28, 1)$ & $158$ & $C_{152}(7, 1; 75, 1; 76, 2; 77, 2)$ \\
$37$ & $C_{34}(16, 18, 19)$ & $200$ & $C_{194}(92, 2; 98, 2; 106, 1; 157, 1)$ \\
$38$ & $C_{31}(11, 2; 16, 2; 18, 1; 23, 1; 26, 1)$ & $278$ & $C_{272}(14, 1; 134, 1; 136, 2; 138, 2)$ \\
$44$ & $C_{36}(16, 2; 19, 2; 22, 2; 35, 2)$ & $296$ & $C_{289}(138, 1; 140, 2; 145, 2; 149, 2)$ \\
$48$ & $C_{39}(10, 2; 13, 1; 21, 2; 27, 3; 29, 1)$ & $342$ & $C_{338}(168, 1; 170, 2; 286, 1)$ \\
$50$ & $C_{42}(14, 1; 21, 2; 22, 3; 30, 1; 36, 1)$ & $420$ & $C_{417}(195, 1; 209, 2)$ \\
$51$ & $C_{49}(23, 25)$ & $466$ & $C_{457}(35, 2; 58, 2; 68, 2; 226, 2; 239, 1)$ \\
$54$ & $C_{51}(26, 2; 30, 1)$ & & \\
\end{tabular}
\caption{Constructions of subcubic trees of order $n$ that are both TI and ETI.}
\label{singular_constructions}
\end{table}

\begin{figure}[t]
\centering
\begin{tikzpicture}[scale=1.0]
\tikzstyle{vertex}=[draw,circle,minimum size=4pt,inner sep=1pt,fill=black]
\tikzstyle{edge}=[draw,thick]
\tikzstyle{dedge}=[draw,thick,dashed]

\node[vertex] (0) at (0, 0) {};
\node[vertex] (1) at (1, 0) {};
\node[vertex] (2) at (2, 0) {};
\node[vertex] (3) at (3, 0) {};
\node[vertex] (4) at (4, 0) {};
\node[vertex] (5) at (5, 0) {};
\node[vertex] (6) at (6, 0) {};
\node[vertex] (7) at (7, 0) {};
\node[vertex] (8) at (8, 0) {};

\node[vertex] (9) at (4, -1) {};
\node[vertex] (10) at (4, 1) {};
\node[vertex] (11) at (3, 1) {};
\node[vertex] (12) at (2, 1) {};
\node[vertex] (13) at (6, -1) {};

\path[edge] (0) -- (1);
\path[edge] (1) -- (2);
\path[edge] (2) -- (3);
\path[edge] (3) -- (4);
\path[edge] (4) -- (5);
\path[edge] (5) -- (6);
\path[edge] (6) -- (7);
\path[edge] (7) -- (8);

\path[edge] (4) -- (9);
\path[edge] (4) -- (10);
\path[edge] (10) -- (11);
\path[edge] (11) -- (12);
\path[edge] (6) -- (13);
\end{tikzpicture}
\caption{The unique TI tree of even order at most $14$.}
\label{ti_14_fig}
\end{figure}
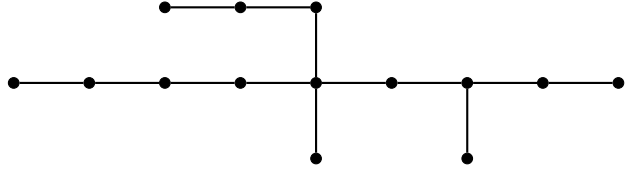

\section{Conclusion}
\label{sec:conclusion}

With Theorem~\ref{th:main} in hand, a broad range of order realizability problems can be settled with little additional effort. Specifically, one may restrict attention to TI graphs, ETI graphs, or graphs that are simultaneously TI and ETI; impose subcubic, subquartic, or no degree constraints; and consider either arbitrary graphs or trees. The realizable orders for each resulting graph class can then be determined as stated in the following result.

\begin{theorem}\label{brutal_th}
The order realizability for general, subquartic, and subcubic graphs and trees with respect to the TI, ETI, and simultaneous TI and ETI properties is exactly as specified in Table~\ref{existence_problems}.
\end{theorem}
\begin{proof}
    Theorem \ref{th:main} implies the existence of a graph of order $n$ in each of the specified classes, for any $n \in \{ 9, 11, 13 \}$ or $n \ge 15$. Also, it is straightforward to verify by computer that there exists no TI or ETI graph of order below $7$, besides the trivial cases, which are excluded.

    Now, assume that $n \in \{ 7, 8, 10, 12, 14 \}$. As established in Section \ref{sec:rare}, there is no ETI graph of order $7$, while $S(1, 2, 3)$ is a subcubic TI tree of order $7$, settling the case $n = 7$. For $n \in \{ 8, 10, 12 \}$, the computational results from \cite{StoDam2026} yield the nonexistence of a TI tree of order $n$. Thus, it suffices to construct a subcubic graph of order $n$ that is both TI and ETI, and a subcubic ETI tree of order $n$. The required subcubic graphs that are simultaneously TI and ETI are given in Figs.~\ref{fig:smallest_ti_eti}, \ref{fig:unicyclic_cool_1} and \ref{fig:unicyclic_cool_2}. Furthermore, it is straightforward to verify that the graphs $S(1, 2, 4)$, $S(2, 3, 4)$ and $S(1, 4, 6)$ are subcubic ETI trees of orders $8$, $10$ and~$12$, respectively, which settles the case $n \in \{ 8, 10, 12 \}$.

    Finally, assume that $n = 14$. As already mentioned, the only TI tree of order $14$ is the tree from Fig.~\ref{ti_14_fig}, and this tree is ETI and subquartic, but not subcubic. To complete the proof, it suffices to observe that the graph from Fig.~\ref{fig:unicyclic_cool_3} is a subcubic graph of order $14$ that is both TI and ETI, while $S(1, 5, 7)$ is a subcubic ETI tree of order $14$.
\end{proof}

\begin{table}[t]
\centering
\begin{tabular}{l||c|c|c|c|c|c|c|c|c|c}
& 1--6 & 7 & 8 & 9 & 10 & 11 & 12 & 13 & 14 & 15-- \\
\hline \hline
TI graph & $\times$ & $\checkmark$ & $\checkmark$ & $\checkmark$ & $\checkmark$ & $\checkmark$ & $\checkmark$ & $\checkmark$ & $\checkmark$ & $\checkmark$\\
\hline
Subquartic TI graph & $\times$ & $\checkmark$ & $\checkmark$ & $\checkmark$ & $\checkmark$ & $\checkmark$ & $\checkmark$ & $\checkmark$ & $\checkmark$ & $\checkmark$\\
\hline
Subcubic TI graph & $\times$ & $\checkmark$ & $\checkmark$ & $\checkmark$ & $\checkmark$ & $\checkmark$ & $\checkmark$ & $\checkmark$ & $\checkmark$ & $\checkmark$\\
\hline
ETI graph & $\times$ & $\times$ & $\checkmark$ & $\checkmark$ & $\checkmark$ & $\checkmark$ & $\checkmark$ & $\checkmark$ & $\checkmark$ & $\checkmark$\\
\hline
Subquartic ETI graph & $\times$ & $\times$ & $\checkmark$ & $\checkmark$ & $\checkmark$ & $\checkmark$ & $\checkmark$ & $\checkmark$ & $\checkmark$ & $\checkmark$\\
\hline
Subcubic ETI graph & $\times$ & $\times$ & $\checkmark$ & $\checkmark$ & $\checkmark$ & $\checkmark$ & $\checkmark$ & $\checkmark$ & $\checkmark$ & $\checkmark$\\
\hline
TI/ETI graph & $\times$ & $\times$ & $\checkmark$ & $\checkmark$ & $\checkmark$ & $\checkmark$ & $\checkmark$ & $\checkmark$ & $\checkmark$ & $\checkmark$\\
\hline
Subquartic TI/ETI graph & $\times$ & $\times$ & $\checkmark$ & $\checkmark$ & $\checkmark$ & $\checkmark$ & $\checkmark$ & $\checkmark$ & $\checkmark$ & $\checkmark$\\
\hline
Subcubic TI/ETI graph & $\times$ & $\times$ & $\checkmark$ & $\checkmark$ & $\checkmark$ & $\checkmark$ & $\checkmark$ & $\checkmark$ & $\checkmark$ & $\checkmark$\\
\hline\hline
TI tree & $\times$ & $\checkmark$ & $\times$ & $\checkmark$ & $\times$ & $\checkmark$ & $\times$ & $\checkmark$ & $\checkmark$ & $\checkmark$\\
\hline
Subquartic TI tree & $\times$ & $\checkmark$ & $\times$ & $\checkmark$ & $\times$ & $\checkmark$ & $\times$ & $\checkmark$ & $\checkmark$ & $\checkmark$\\
\hline
Subcubic TI tree & $\times$ & $\checkmark$ & $\times$ & $\checkmark$ & $\times$ & $\checkmark$ & $\times$ & $\checkmark$ & $\times$ & $\checkmark$\\
\hline
ETI tree & $\times$ & $\times$ & $\checkmark$ & $\checkmark$ & $\checkmark$ & $\checkmark$ & $\checkmark$ & $\checkmark$ & $\checkmark$ & $\checkmark$\\
\hline
Subquartic ETI tree & $\times$ & $\times$ & $\checkmark$ & $\checkmark$ & $\checkmark$ & $\checkmark$ & $\checkmark$ & $\checkmark$ & $\checkmark$ & $\checkmark$\\
\hline
Subcubic ETI tree & $\times$ & $\times$ & $\checkmark$ & $\checkmark$ & $\checkmark$ & $\checkmark$ & $\checkmark$ & $\checkmark$ & $\checkmark$ & $\checkmark$\\
\hline
TI/ETI tree & $\times$ & $\times$ & $\times$ & $\checkmark$ & $\times$ & $\checkmark$ & $\times$ & $\checkmark$ & $\checkmark$ & $\checkmark$\\
\hline
Subquartic TI/ETI tree & $\times$ & $\times$ & $\times$ & $\checkmark$ & $\times$ & $\checkmark$ & $\times$ & $\checkmark$ & $\checkmark$ & $\checkmark$\\
\hline
Subcubic TI/ETI tree & $\times$ & $\times$ & $\times$ & $\checkmark$ & $\times$ & $\checkmark$ & $\times$ & $\checkmark$ & $\times$ & $\checkmark$\\
\end{tabular}
\caption{Order realizability for various graph classes, with TI/ETI referring to graphs that are simultaneously TI and ETI. A check mark indicates that there exists a graph in the specified class with the prescribed order, while a cross mark indicates that it does not exist. The trivial graph $K_1$ is excluded for the TI property, while the trivial graphs $K_1$ and $K_2$ are excluded for the ETI property.}
\label{existence_problems}
\end{table}

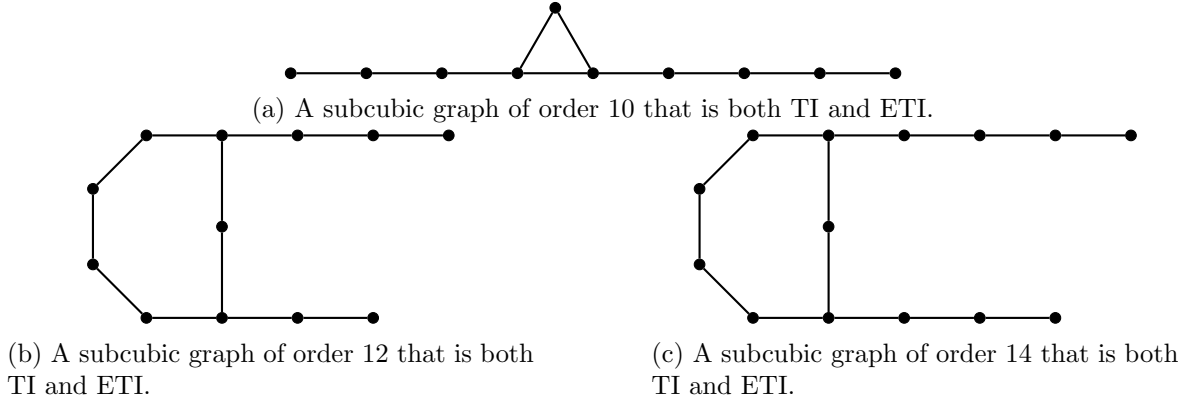
\begin{figure}[t]
\centering

\begin{subfigure}[b]{0.95\textwidth}
\centering
\begin{tikzpicture}[scale=1.0]
\tikzstyle{vertex}=[draw,circle,minimum size=4pt,inner sep=1pt,fill=black]
\tikzstyle{edge}=[draw,thick]
\tikzstyle{dedge}=[draw,thick,dashed]

\node[vertex] (0) at (0, 0) {};
\node[vertex] (1) at (1, 0) {};
\node[vertex] (2) at (2, 0) {};
\node[vertex] (3) at (3, 0) {};
\node[vertex] (4) at (4, 0) {};
\node[vertex] (5) at (5, 0) {};
\node[vertex] (6) at (6, 0) {};
\node[vertex] (7) at (7, 0) {};
\node[vertex] (8) at (8, 0) {};
\node[vertex] (9) at (3.5, 0.866) {};

\path[edge] (0) -- (1);
\path[edge] (1) -- (2);
\path[edge] (2) -- (3);
\path[edge] (3) -- (4);
\path[edge] (4) -- (5);
\path[edge] (5) -- (6);
\path[edge] (6) -- (7);
\path[edge] (7) -- (8);
\path[edge] (9) -- (3);
\path[edge] (9) -- (4);
\end{tikzpicture}
\caption{A subcubic graph of order $10$ that is both TI and ETI.}
\label{fig:unicyclic_cool_1}
\end{subfigure}
\\
\begin{subfigure}[b]{0.45\textwidth}
\centering
\begin{tikzpicture}[scale=1.0]
\tikzstyle{vertex}=[draw,circle,minimum size=4pt,inner sep=1pt,fill=black]
\tikzstyle{edge}=[draw,thick]
\tikzstyle{dedge}=[draw,thick,dashed]

\node[vertex] (0) at (0, 0) {};
\node[vertex] (1) at (0, 1.207) {};
\node[vertex] (2) at (-1, 1.207) {};
\node[vertex] (3) at (-1.707, 0.5) {};
\node[vertex] (4) at (0, -1.207) {};
\node[vertex] (5) at (-1, -1.207) {};
\node[vertex] (6) at (-1.707, -0.5) {};

\node[vertex] (7) at (1, 1.207) {};
\node[vertex] (8) at (2, 1.207) {};
\node[vertex] (9) at (3, 1.207) {};
\node[vertex] (10) at (1, -1.207) {};
\node[vertex] (11) at (2, -1.207) {};

\path[edge] (0) -- (1);
\path[edge] (1) -- (2);
\path[edge] (2) -- (3);
\path[edge] (0) -- (4);
\path[edge] (4) -- (5);
\path[edge] (5) -- (6);
\path[edge] (3) -- (6);

\path[edge] (1) -- (7);
\path[edge] (7) -- (8);
\path[edge] (8) -- (9);
\path[edge] (4) -- (10);
\path[edge] (10) -- (11);

\end{tikzpicture}
\caption{A subcubic graph of order $12$ that is both TI and ETI.}
\label{fig:unicyclic_cool_2}
\end{subfigure}
\hfill
\begin{subfigure}[b]{0.45\textwidth}
\centering
\begin{tikzpicture}[scale=1.0]
\tikzstyle{vertex}=[draw,circle,minimum size=4pt,inner sep=1pt,fill=black]
\tikzstyle{edge}=[draw,thick]
\tikzstyle{dedge}=[draw,thick,dashed]

\node[vertex] (0) at (0, 0) {};
\node[vertex] (1) at (0, 1.207) {};
\node[vertex] (2) at (-1, 1.207) {};
\node[vertex] (3) at (-1.707, 0.5) {};
\node[vertex] (4) at (0, -1.207) {};
\node[vertex] (5) at (-1, -1.207) {};
\node[vertex] (6) at (-1.707, -0.5) {};

\node[vertex] (7) at (1, 1.207) {};
\node[vertex] (8) at (2, 1.207) {};
\node[vertex] (9) at (3, 1.207) {};
\node[vertex] (10) at (4, 1.207) {};
\node[vertex] (11) at (1, -1.207) {};
\node[vertex] (12) at (2, -1.207) {};
\node[vertex] (13) at (3, -1.207) {};

\path[edge] (0) -- (1);
\path[edge] (1) -- (2);
\path[edge] (2) -- (3);
\path[edge] (0) -- (4);
\path[edge] (4) -- (5);
\path[edge] (5) -- (6);
\path[edge] (3) -- (6);

\path[edge] (1) -- (7);
\path[edge] (7) -- (8);
\path[edge] (8) -- (9);
\path[edge] (9) -- (10);
\path[edge] (4) -- (11);
\path[edge] (11) -- (12);
\path[edge] (12) -- (13);

\end{tikzpicture}
\caption{A subcubic graph of order $14$ that is both TI and ETI.}
\label{fig:unicyclic_cool_3}
\end{subfigure}

\caption{Subcubic graphs of order $n$ that are both TI and ETI, for $n \in \{ 10, 12, 14 \}$.}
\label{fig:unicyclic_cool}
\end{figure}

In addition, a \texttt{SageMath} script available in~\cite{GitHub} can be used to computationally verify the validity of all the constructions used in the proof of Theorem~\ref{brutal_th}. This result yields various direct corollaries concerning both the TI and ETI properties of graphs and trees, such as the following.

\begin{corollary}
    For any $n \ge 7$, there exists a chemical TI graph of order $n$.
\end{corollary}
\begin{corollary}
    For any $n \in \{ 7, 9, 11 \}$ or $n \ge 13$, there exists a chemical TI tree of order $n$.
\end{corollary}
\begin{corollary}
    For any $n \ge 8$, there exists a chemical ETI tree of order $n$.
\end{corollary}
\begin{corollary}
    For any $n \ge 8$, there exists a chemical graph of order $n$ that is simultaneously TI and ETI.
\end{corollary}

We believe that a natural direction for future research is to compare the edge Wiener dimension with the Wiener dimension, since this would give more insight into the relationship between the TI and ETI properties. With this in mind, we propose the following problem.

\begin{problem}\label{w-ti-eti}
Investigate the possible range of values for the difference between the Wiener dimension 
and the edge Wiener dimension within the class of graphs of a given order $n$, 
and estimate how large this difference can become. What can 
be said if we restrict to trees of order~$n$?
\end{problem}

\noindent
As a special case of Problem \ref{w-ti-eti}, this following problem may also be interesting.

\begin{problem}
Determine the upper bound of the Wiener dimension of all ETI graphs in a fixed class such as trees or other graphs with one or more structural parameters,  and the upper bound of the edge Wiener dimension of all TI graphs in this class.
\end{problem}

\section*{Acknowledgments}

K.~Xu was partially supported by NNSF of China (Grant No.\ 12271251). I.\ Damnjanovi\'{c} was supported by the Ministry of Science, Technological Development and Innovation of the Republic of Serbia, grant number 451-03-34/2026-03/200102, and the Science Fund of the Republic of Serbia, grant \#6767, Lazy walk counts and spectral radius of threshold graphs --- LZWK. S.~Klav\v{z}ar was supported by the Slovenian Research and Innovation Agency (ARIS) under the grants P1-0297, N1-0355, N1-0285, N1-0431, and J1-70045.

\end{document}